\documentclass{amsart}
\usepackage{pb-diagram}
\usepackage{amsfonts}
\usepackage{amsthm}
\usepackage{amsmath}
\usepackage{mathtools}
\usepackage{mathrsfs}
\usepackage{amssymb}
\usepackage{bm}
\usepackage{appendix}
\usepackage[all,cmtip,2cell]{xy}
\usepackage{comment}

\usepackage{autobreak}
\usepackage{enumitem}
\usepackage[hypertexnames=false]{hyperref} 
\usepackage[capitalize]{cleveref}
\crefname{equation}{}{}
\crefname{enumi}{}{}

\usepackage[\ifdefined\enabletodonotes\else disable\fi]{todonotes}

\usepackage{tikz}
\usetikzlibrary{positioning}
\usetikzlibrary{arrows.meta}

\numberwithin{equation}{section}
\newtheorem{thm}{Theorem}[subsection]
\newtheorem{prop}[thm]{Proposition}
\newtheorem{cor}[thm]{Corollary}
\newtheorem{lem}[thm]{Lemma}

\theoremstyle{definition}
\newtheorem{defn}[thm]{Definition}
\newtheorem{const}[thm]{Construction}

\theoremstyle{remark}
\newtheorem{rem}[thm]{Remark}

\author[M. Noda]{Masai Noda}
\address{%
  Department of Science and Technology,
  Graduate school of Medicine, Science and Technology,
  Shinshu University,
  Matsumoto, Nagano 390-8621, Japan
}
\email{24hs604b@shinshu-u.ac.jp}
\title[Adjoint pairs between categories of sheaves]{An adjunction of the categories of sheaves related to a topology and a diffeology
}

\subjclass[2020]{Primary 58A03; Secondary 58A40, 57P05, 18F10, 18F20}
\keywords{ Diffeology, sheaf, Grothendieck topology, topos.}
\thanks{}

\PassOptionsToPackage{final}{showkeys}
\PassOptionsToPackage{final}{showlabels}
\usepackage[final]{showkeys}
\begin{document}

\begin{abstract}
There exists an adjoint pair $(D, C)$ of functors between the category of diffeological
spaces and that of topological spaces. In this article, by using the functors, we
introduce two pairs of functors between the categories of sheaves on a diffeological space and on 
a topological space. Then, the adjointness of the novel functors up to the functors
induced by the unit and counit of $C$ and $D$ is clarified. 
There is no natural functor between sites obtained by a diffeological space and  a topological space. 
Therefore, the adjunctions on the categories of sheaves are elaborated without applying the general theory of topoi. 
\end{abstract}

\maketitle

\section{Introduction}
A {\it diffeological space} is a set endowed with a family of maps from open subsets of Euclidean spaces to the underlying set. The object is indeed regarded as a generalization of a smooth manifold. However, unlike a manifold, a general diffeological space is constructed by attaching domains of {\it plots} possibly with different dimensions. The category $\mathsf{Diff}$ of diffeological spaces is complete, cocomplete and Cartesian closed. Thus, we have naturally quotient spaces, mapping spaces and adjunction spaces in $\mathsf{Diff}$.

A sheaf is an important and powerful tool for extending certain local data to global ones. A sheaf on a diffeological space is introduced by Krepski, Watts and Worbert \cite{KWW} and Dehghan Nezhad and Ahmadi \cite{AA} individually. In the both definitions, the category of {\it plots} is used. 
The sheaf in the sense in \cite{KWW} is regarded as a family of sheaves on the domains of plots. On the other hand,
the sheaf described in  \cite{AA} is defined with {\it a compatible family} of plots which satisfies the usual sheaf condition. 
In what follows, a sheaf on a diffeological space is that in the sense in \cite{KWW} unless otherwise specified. 


It is known that there exists an adjunction $D:\mathsf{Diff}\rightleftarrows\mathsf{Top}:C$ between $\mathsf{Diff}$ and the category  $\mathsf{Top}$ of topological spaces such that the adjoints are identities on the underlying sets; 
see Section \ref{sect:diff} and the subsequent comment for more details. 
In this article, we construct functors $\widetilde{C}$ and $\widetilde{D}$ associated with $C$ and $D$ between the categories of $\mathcal{C}$-valued sheaves on a diffeological space and a topological space, where $\mathcal{C}$ is an appropriate category.

In order to describe our main results more precisely, we prepare terminology and notation. 
Let $Sh_{\mathsf{Diff}}^{\mathcal{C}}\mathcal{D}_{X}$ and $Sh_{\mathsf{Top}}^{\mathcal{C}}Open(Y)$ be
the category of the $\mathcal{C}$-valued sheaves on a diffeological space $(X,\mathcal{D}_{X})$ and that on a topological space $(Y, Open(Y))$, respectively.  We write $(X, Open(DX))$ and $(Y,\mathcal{D}_{CY})$ for $D(X, \mathcal{D}_{X})$ and $C(Y, Open(Y))$, respectively. 
We observe that there is no natural functor between the Grothendieck site of open subsets and that of plots. Therefore,  the formal method in sheaf theory would not be applicable to obtain the functors between the category of sheaves mentioned above.  
Nevertheless, the adjunction $(D, C)$ enables us to obtain adjoints between the categories of sheaves on a diffeological space and a topological space {\it up to} functors obtained by the unit and counit of $(D, C)$. 

Our main results in this manuscript are summarized as follows. 

\medskip
\noindent 
$\mathbf{Theorem \ A \ (Theorems \hspace{3pt}\ref{thm:main,eta.ver.}\hspace{2pt}and\hspace{2pt}\ref{thm:main,epsilon.ver}})$\emph{
There exist functors $\widetilde{D} : Sh_{\mathsf{Diff}}^{\mathcal{C}}\mathcal{D}_X \to Sh_{\mathsf{Top}}^{\mathcal{C}}Open(DX)$ and 
$\widetilde{C} : Sh_{\mathsf{Top}}^{\mathcal{C}}Open(Y) \to Sh_{\mathsf{Diff}}^{\mathcal{C}}\mathcal{D}_{CY}$ which give rise to 
adjoint pairs}
$$
\xymatrix@C=10pt@R=10pt{
Sh_{\mathsf{Diff}}^{\mathcal{C}}\mathcal{D}_{X}\ar[rr]^(.45){\widetilde{D}}&&Sh_{\mathsf{Top}}^{\mathcal{C}}Open(DX)\ar[ld]^{\widetilde{C}} \ \ \   \text{and}\\
&Sh_{\mathsf{Diff}}^{\mathcal{C}}\mathcal{D}_{CDX}\ar[lu]^{\eta^{*}_{X}}\ar@{}[u]|{\top} & \\
Sh_{\mathsf{Top}}^{\mathcal{C}}Open(Y)\ar[rr]^{\widetilde{C}}&&Sh_{\mathsf{Diff}}^{\mathcal{C}}\mathcal{D}_{CY}\ar[ld]^{\widetilde{D}}\\
&Sh_{\mathsf{Top}}^{\mathcal{C}}Open(DCY)\ar[lu]^{\varepsilon_{X*}}\ar@{}[u]|{\perp}, &  
}
$$
\emph{where $\eta^*_X$  and  $\varepsilon_{X*}$ are functors induced by the unit 
$\eta : \text{\em id}_{\mathsf{Diff}} \Rightarrow CD$ and 
the counit $\varepsilon : DC \Rightarrow \text{\em id}_{\mathsf{Top}}$, respectively.  
}

\medskip
The adjunction $(D, C)$ satisfies the condition that $CDC=C$ and $DCD=C$; see Remark \ref{rem:C-D}. 
The fact enables us to deduce that the functors $\widetilde{C}$ and $\widetilde{D}$ indeed give adjoint pairs. 

\medskip
\noindent
$\mathbf{Corollary \ B \ (Corollary\hspace{3pt}\ref{cor:adjunction}})$ \emph{One has adjoint pairs}
\begin{eqnarray}\label{eq:Triangle_1}
\xymatrix@C=55pt@R=20pt{
Sh_{\mathsf{Top}}^{\mathcal{C}}Open(DX)\ar@<1.0ex>[r]^{\widetilde{C}}_{\perp}&Sh_{\mathsf{Diff}}^{\mathcal{C}}\mathcal{D}_{CDX}\ar@<1.7ex>[l]^{\widetilde{D}}  \ \ \ \  \text{and} \\
Sh_{\mathsf{Diff}}^{\mathcal{C}}\mathcal{D}_{CY}\ar@<.8ex>[r]^(.4){\widetilde{D}}_(.4){\top}&Sh_{\mathsf{Top}}^{\mathcal{C}}Open(DCY).\ar@<1.7ex>[l]^(.6){\widetilde{C}}  \ \ \ \ \ 
}
\end{eqnarray}

 \begin{rem}
Roughly speaking, the functors $\widetilde{C}$ and $\widetilde{D}$ are defined with a colimit construction together with the sheafification and a limit construction, respectively; see Definitions \ref{defn:overlineC}, \ref{defn:widetildeC}, \ref{defn:top.sheafification}, \ref{defn:sheafification} and \ref{defn:wildetildeD}.  While the proof of Theorem \ref{thm:main,epsilon.ver} is the same as that of Theorem \ref{thm:main,eta.ver.}, we prove the result observing the change of the order of colimits and limits.
\end{rem}


The rest of the manuscript is organized as follows. 
Section \ref{sect:diff} begins with the definition of a diffeological space and some constructions in diffeology. 
The functors $C:\mathsf{Top}\to\mathsf{Diff}$ and $D:\mathsf{Diff}\to\mathsf{Top}$ aforementioned are introduced in this section by following \cite{IZ, CSW, Sh}. 
Section \ref{sect:diff-sheaves} recalls the definition of the diffeological sheaves of \cite{KWW} and 
the sheafification functor on $\mathsf{Diff}$ is described; see Propositions \ref{prop:sheafification0} and \ref{prop:sheafification}. 
The constructions of the functors $\widetilde{C}$ and $\widetilde{D}$ in Theorem A 
are explained in Section \ref{sect:CD}. 
Section \ref{sect:Proofs} is devoted to proving Theorem A and Corollary B. In Appendix A, we introduce a Grothendieck topos associated with a diffeological space 
defined in \cite{AA}.
It is proved that the topos is isomorphic to the $\mathsf{Sets}$-valued sheaf in \cite{KWW} whenever the underlying diffeological spaces coincide with each other. 
Appendix B is dedicated to proving Proposition \ref{prop:sheafification}. 

\begin{rem}
As seen in Appendix \ref{sect:AppA}, if $\mathcal{C}$ is the category of sets $\mathsf{Sets}$, the category $Sh^{\mathsf{Sets}}_{\mathsf{Diff}}\mathcal{D}_X$ is regarded as a Grothendieck topos. Then, the general theory of the shefification \cite[V. 5]{SMIM} is applicable to $Sh_{\mathsf{Diff}}^{\mathsf{Sets}}\mathcal{D}_X$. However, 
in order to define and consider the functor $\widetilde{C}$ mentioned above, we give explicitly the shefification functor in our context; see Proposition \ref{prop:sheafification}. 
\end{rem}

\section{An overview of diffeology}\label{sect:diff}
In this section, we recall the definition of diffeology and some basic properties that are used in this paper. 
We refer the reader to the book \cite{IZ} of Iglesias-Zemmour on diffeology for more details.
\subsection{Diffeology}
We begin with the definition of 
a diffeological spaces.
\begin{defn}
 Let $X$ be a set. A {\it parametrization} is a set-theoretic map to $X$ from an open subset of $\mathbb{R}^n$ for some $n$.  
 The set of parametrizations taking values in $X$ is denoted by $Param(X)$. 
 A {\emph{diffeology}} $\mathcal{D}_{X}$ of $X$ is a subset of $Param(X)$ satisfying the following conditions (D1), (D2) and (D3). 
\begin{enumerate}
\item[(D1)] ({\em{Covering axiom}}) All constant maps are in $\mathcal{D}_{X}$
\item[(D2)] ({\em{Locality axiom}}) Given a parametrization $p : U \to X$, if there exists an open cover $\{U_\lambda\}_{\lambda}$ of $U$ such that the restriction $p|_{U_\lambda}$ is in $\mathcal{D}_X$ for each $\lambda$, then $p$ is in $\mathcal{D}_{X}$. 
\item[(D3)] ({\em{Compatibility axiom}}) Let $p:U_{p}\to X$ be in $\mathcal{D}_{X}$ and $f:U_{f}\to U_{p}$ a smooth map from an open subset of $\mathbb{R}^m$ for some $m$. 
Then, the composite $p\circ f$ is in 
$\mathcal{D}_{X}$.
\end{enumerate}
\end{defn}
An element of $\mathcal{D}_{X}$ is called a $\emph{plot}$ of $X$. Note that $\mathcal{D}_{X}$ always includes a unique empty plot 
$\mathfrak{e}:\phi\to X$. 
In what follows, 
the domain of a plot $p$ is denoted by $U_{p}$. 
A pair $(X, \mathcal{D}_{X})$ of a set and a diffeology is called a {\emph{diffeological space}}.

\begin{defn}
(Smooth map) Let $(X,\mathcal{D}_{X})$ and $(Y,\mathcal{D}_{Y})$ be diffeological spaces. A map $f:X\to Y$ is $smooth$ if for each plot $p\in\mathcal{D}_{X}$, the composite $f\circ p$ is a plot of $Y$.  
\end{defn}
It is readily seen that diffeological spaces and smooth maps give a category. In what follows, the category is denoted by 
$\mathsf{Diff}$. An isomorphism in $\mathsf{Diff}$ is called a {\it diffeomorphism}.
\begin{defn}
We recall some basic constructions in $\mathsf{Diff}$. Let $(X,\mathcal{D}_{X})$ be a diffeological space and $\{(X_{\alpha},\mathcal{D}_{X_{\alpha}})\}_{\alpha}$ a family of  diffeological spaces.

(i) 
Let $A$ be a subset of $X$. The {\it subset diffeology} of $A$ is defined by the collection of plots in $\mathcal{D}_{X}$ whose images are in $A$.

(ii) 
The {\it sum diffeology} on $\coprod_{\alpha}X_{\alpha}$ is the set of parametrizations $p$ of $\coprod_{\alpha}X_{\alpha}$ which satisfy the condition 
that, for all $u\in U_{p}$, there exist an open neighborhood $V$ of $u$, $\alpha$ and $q_{\alpha}\in\mathcal{D}_{X_{\alpha}}$ such that $p|_{V}=q_{\alpha}\circ i$, where $i:V\hookrightarrow U_{q_{\alpha}}$ is  an inclusion map.

(iii) 
The {\it product diffeology} on $\prod_{\alpha}X_{\alpha}$ is the set of parametrizations $p$ of $\prod_{\alpha}X_{\alpha}$  each of which the composition $pr_\alpha \circ p$ with the projection $pr_\alpha$ in the $\alpha$th factor is a plot of $X_\alpha$. 

(iv) 
Given an equivalence relation $\sim$ on the underlying set of a diffeological space $(X, \mathcal{D}_X)$, the {\it quotient diffeology} of $X/\!\sim$ consists of the collection of $p\in Param(X/\!\sim)$ that for each $u\in U_{p}$, there exists a neighborhood $V$ of $u$ and $q\in\mathcal{D}_{X}$ such that $p|_{V}=\pi\circ q$, where 
 $\pi$ is the quotient map $X \to X/\!\sim$. 
\end{defn}
\begin{thm}\text{\em (\cite{CSW, BH})}
\text{\rm (Colimit/Limit in $\mathsf{Diff}$)} The category $\mathsf{Diff}$ is both complete and cocomplete. 
\end{thm} 

We stress that the category of (possibly infinite dimensional) manifolds modeled on locally convex spaces with convenient calculus (\cite{K-M}) embeds in the category $\mathsf{Diff}$; see \cite[Lemma 2.5]{Kihara}. 

For a topological space $(X,Open(X))$, let $C(X, Open(X))$ be the diffeological space with the same underlying set $X$ and with all continuous maps from open subsets of Euclidian spaces into $X$ as plots. 
It is readily seen that a continuous map $f:X\to Y$ induces a smooth map $C(X,Open(X))\to C(Y,Open(Y))$. Then, we have a functor $C:\mathsf{Top}\to\mathsf{Diff}$, where 
$\mathsf{Top}$ denotes the category of topological spaces. 

For a diffeological space $(X,\mathcal{D}_X)$, there exists the finest topology such that the plots are continuous, which is called the \emph{D-topology} of $X$. The open sets of the topology are characterized by the following property: A subset $A\subset X$ is open for the $D$-topology if and only if, for every plot $p\in \mathcal{D}_X$, $p^{-1}(A)$ is open. The open sets of the $D$-topology are called \emph{D-open sets}. A smooth map $(X,\mathcal{D}_X)\to (Y,\mathcal{D}_{Y})$ gives a continuous when $X$ and $X'$ are equipped with the $D$-topology. Thus we have a functor $D:\mathsf{Diff}\to\mathsf{Top}$.

\begin{rem}\label{rem:C-D}
The functor $D$ is the left adjoint to $C$. Moreover, these functor satisfy $CDC=C$ and $DCD=D$; see \cite[Proposition 3.1]{Sh} and \cite[Proposition 3.3]{CSW} for more details. 
\end{rem}
\begin{defn}\label{defn:category of plots}
 Let $(X,\mathcal{D}_{X})$ be a diffeological space.  The {\it category of plots} is a category comprising plots of $X$ as objects and morphisms 
 $f : p\to q$ each of which gives a commutative triangle
$$
\xymatrix@C20pt@R10pt{
&X&\\
U_{p}\ar[ru]^{p}\ar[rr]_{f}&&U_{q}, \ar[lu]_{q}
}
$$
where $p$ and $q$ are plots and $f: U_{p}\to U_{q}$ is a smooth map.
\end{defn}
By abuse of notation, we also write $\mathcal{D}_{X}$ for the category of plots.

\section{A sheaf theory on diffeological spaces}\label{sect:diff-sheaves}
We begin by recalling the definition of a sheaf on a diffeological space. 
The sheaf is defined with the category of plots. 
\subsection{Sheaves on a diffeological space}
\begin{defn}
\label{defn:presheaf}
Let  $\mathcal{D}_{X}$ be the category of plots of a diffeological space $(X,\mathcal{D}_X)$. Let $\mathcal{C}$ denote one of categories $\mathsf{Sets}$ of sets,  $\mathsf{Grp}$ of groups, $\mathsf{Ab}$ of abelian groups and $\mathsf{Vect}$ of vector spaces. A {\it $\mathcal{C}$-valued presheaf} is a functor $F:\mathcal{D}_{X}^{op}\to\mathcal{C}$. Let $PSh^{\mathcal{C}}_{\mathsf{Diff}}\mathcal{D}_{X}$ be the category consisting of  $\mathcal{C}$-valued presheaves on $\mathcal{D}_{X}$ as objects and  natural transformations as morphisms.
\end{defn}

The following definition of a sheaf is due to Krepski, Watts and Wolbert \cite{KWW}. 

\begin{defn} 
\label{defn:sheaf}
 Let $F:\mathcal{D}_{X}^{op}\to\mathcal{C}$ be a $\mathcal{C}$-valued presheaf. 
 For a plot $p \in \mathcal{D}_{X}$, let $ Open(U_{p})$ be the category consisting of open subsubsets of $U_{p}$ as objects and inclusion maps as morphisms. 
For each plot $p:U_{p}\to X$, one obtains a presheaf $F|_{p}:Open(U_{p})^{op}\to\mathcal{C}$  in the classical sense defined by $F|_p(V)=F(p\circ\iota)$, where $\iota : V \hookrightarrow U_p $ is the  inclusion map. 

A $\mathcal{C}$-valued presheaf $F:\mathcal{D}_{X}^{op}\to\mathcal{C}$ is a {\it $\mathcal{C}$-valued sheaf} if the presheaf $F|_{p}$ is a sheaf on $Open(U_{p})$ for every plot of $p\in\mathcal{D}_{X}$ and $F(\mathfrak{e})$ is  the terminal object of $\mathcal{C}$, where $\mathfrak{e}$ is an empty plot of $\mathcal{D}_X$; see Section \ref{sect:diff}.
The full subcategory of $PSh_{\mathsf{Diff}}^{\mathcal{C}}\mathcal{D}_X$  of the $\mathcal{C}$-valued sheaves is denoted by $Sh_{\mathsf{Diff}}^{\mathcal{C}}\mathcal{D}_{X}$. 
\end{defn}

We recall the sheafification functor for preshaves on a topological space. 

\begin{defn}\label{defn:top.sheafification}
 (\cite[6.17]{S}) Let $X$ be a topological space and $Open(X)$ the category consisting of open subsets of $X$ and  inclusions. For a given presheaf  $F:Open(X)^{op}\to\mathcal{C}$, sheafification of $F$ is defined by 
$$
a_XF(U):=\{(s_u)_u\in\prod_{u\in U}F_u| \  (**) \ \},
$$
where $F_u$ is the stalk of $F$ at $u$ and $(**)$ denotes the condition: 
\begin{itemize}
\item[] For every $w\in U$, there exists an 
open neighborhood $w\in \mathcal{U}^w\subset U$, and a section $\sigma^w\in F(\mathcal{U}^w)$ such that for all $v\in \mathcal{U}^w$, we have $s_v=[\sigma^w]_v$ in $F_v$.
\end{itemize}
The element $[\sigma^w]_v$ denotes an equivalence class of $\sigma^w$ in $F_v$. For each inclusion $\iota:U'\hookrightarrow U$, the map $a_X(\iota):a_XF(U)\to a_XF(U')$ is defined by  $a_X(\iota)\hspace{2pt}((s_u)_{u\in U})=(s_u)_{u\in U'}$. We have the presheaf $a_XF:Open(X)^{op}\to \mathcal{C}$ which satisfies the sheaf condition.
Then, one has a functor $a_X:PSh_{\mathsf{Top}}^{\mathcal{C}}Open(X)\to Sh_{\mathsf{Top}}^{\mathcal{C}}Open(X)$, which is called {\it the sheafification functor}.  
\end{defn}

The functor $a_X$ is the left adjoint to the inclusion functor   $i_X:Sh_{\mathsf{Top}}^{\mathcal{C}}Open(X)\hookrightarrow PSh_{\mathsf{Top}}^{\mathcal{C}}Open(X)$; see, for example, \cite[6.17]{S} and \cite[II, Corollary 4]{SMIM}.

\begin{thm}\label{thm:univ.sheafification}
$($The universality of sheafification functor for presheaves on $Open(X))$ Let $X$ be a topological space and $F\in Sh_{\mathsf{Top}}^{\mathcal{C}}Open(X)$. Let $g_{XP}:P\to a_XP$ be a morphism in $ PSh_{\mathsf{Top}}^{\mathcal{C}}Open(X)$ defined by $(g_{XP})_{U}(s_U)=([s_U]_u)_{u\in U}$, where $P\in PSh_{\mathsf{Top}}^{\mathcal{C}}Open(X)$, $s_U\in P(U)$ and $U\in Open(X)$. Then,  for each $Q\in Sh_{\mathsf{Top}}^{\mathcal{C}}Open(X)$ and morphism $\kappa:P\to i_X(Q)$ in $PSh_{\mathsf{Top}}^{\mathcal{C}}Open(X)$, there exists a unique map $\widetilde{\kappa}:a_{X}P\to i_X(Q)$  which fits in  the commutative diagram 
$$
\xymatrix@R=14pt{
P\ar[rr]^{\kappa}\ar[rd]_{g_{XP}}&&i_X(Q), \\
&a_XP\ar@{}[u]|{\circlearrowright}\ar[ru]_{\widetilde{\kappa}}&
}
$$
where $i_X$ denotes the inclusion functor $i_X:Sh_{\mathsf{Top}}^{\mathcal{C}}Open(X)\hookrightarrow PSh_{\mathsf{Top}}^{\mathcal{C}}Open(X)$.
\end{thm}

See \cite[II]{SMIM} for more details about the sheafification functor of preseaves on a topological space.

\begin{defn}\label{defn:sheafification}
 (The sheafification functor for preshaves on $\mathcal{D}_X$)
 Let $\mathcal{D}_{X}$ be the category of plots on a diffeological space $(X,\mathcal{D}_X)$. An  
 assignment $\mathfrak{a}_{X}: PSh_{\mathsf{Diff}}^{\mathcal{C}}\mathcal{D}_X\to Sh_{\mathsf{Diff}}^{\mathcal{C}}\mathcal{D}_X$ is defined by 
 $$
 \mathfrak{a}_{X}F(p):=a_{U_{p}}F|_{p}(U_{p})
 $$
 for a presheaf $F$ and a plot $p\in\mathcal{D}_X$, 
 where $a_{U_{p}}$ denotes the sheafification functor of preseaves on $U_{p}$; see Definition \ref{defn:top.sheafification}. For each natural transformation $\lambda:F\to G$, a plot $p\in\mathcal{D}_X$ and an element $(a_u)_{u\in U_p}\in \mathfrak{a}_XF(p)$, we define $(\mathfrak{a}_X(\lambda))_p:\mathfrak{a}_XF(p)\to\mathfrak{a}_XG(p)$ by
 $$
 (\mathfrak{a}_X(\lambda))_p((a_{u})_{u\in U_{p}})=(\lambda_{u}(a_{u}))_{u\in U_{p}}
 $$
 for each $p$, where $\lambda_{u} : \text{colim}_{u\in V}F|_{p}(V)\to \text{colim}_{u\in V}G|_{p}(V)$ is the map between the stalks at $u\in U_p$ induced by the universality of the stalk of $F|_p$ at $u$.  
\end{defn} 

The following propositions give the well-definedness of $\mathfrak{a}_X$.

\begin{prop}\label{prop:well-def.sheafification}
For $F\in PSh_{\mathsf{Diff}}^{\mathcal{C}}\mathcal{D}_X$, the functor $\mathfrak{a}_{X}F$ defined in Definition \ref{defn:sheafification} is a sheaf on 
$\mathcal{D}_X$.
\end{prop}
\begin{proof}
First, we prove that the functor $\mathfrak{a}_XF$ is a presheaf on $\mathcal{D}_X$. 
By definition, for each morphism $f:p\to q$ of $\mathcal{D}_X$, $V\subset U_q$ and $([s_{\mathcal{V}}]_v)_{v\in V}\in\mathfrak{a}_XF(q|_V)$, where $v\in\mathcal{V}$ is an open subset of $U_q$ and $s_{\mathcal{V}}\in F|_q(\mathcal{V})$, 
 we see that $\mathfrak{a}_XF(f)([s_{\mathcal{V}}]_v)_{v\in V}=([F(f|_{f^{-1}(\mathcal{V})})(s_{\mathcal{V}})]_{u})_{u\in f^{-1}(V)}$. Note that each map between stalks sends a germ $[s_{\mathcal{V}}]_v$ to $[F(f|_{f^{-1}(\mathcal{V})})(s_{\mathcal{V}})]_u$, where $f(u)=v$. Thus,  $\mathfrak{a}_XF:\mathcal{D}_X^{op}\to\mathcal{C}$. 
 
 To show that $\mathfrak{a}_XF$ is a sheaf in the sense  of Definition \ref{defn:sheaf}, it suffices to prove that $(\mathfrak{a}_XF)|_p$ is a sheaf on $U_p$ ; see Definition \ref{defn:top.sheafification}. For any open set $V\subset U_p$, we have 
$$
\xymatrix{
(\mathfrak{a}_XF)|_p(V)=(\mathfrak{a}_XF)(p\circ \iota)=a_VF|_{p\circ\iota}(V)=
a_{U_p}F|_p(V),
}
$$
where $\iota:V\hookrightarrow U_p$. The first two equalities follow from the definitions. 
The last equality is obtained by comparing the two sets. 
In fact, the elements of $a_VF|_{p\circ\iota}(V)$ have the form  $(a_x)_{x\in V}\in\prod_{x\in V}(F|_{p\circ\iota})_x$ satisfying the condition $(**)$ for any $x\in V$. Then, there exists an open subset $x\in\mathcal{U}^x\subset V$ and a section $\sigma^x\in F|_{p\circ\iota}(\mathcal{U}^x)$ such that $[\sigma^x]_w=a_w $ for any $w\in\mathcal{U}^x$; see the property $(**)$ in Definition \ref{defn:top.sheafification}. Observe  that these open subsets $\mathcal{U}^x\subset V$ are also ones of $U_p$ and 
$$
F|_{p\circ\iota}(\mathcal{U}^x)=F(p\circ\iota\circ\iota^x)=F(p|_{\mathcal{U}^x})=F|_p(\mathcal{U}^x), 
$$
where $\iota^x:\mathcal{U}^x\hookrightarrow V$. Therefore, the section $\sigma^x$ is an element of $F|_p(\mathcal{U}^x)$. Thus, $a_VF|_{p\circ\iota}(V)\subset a_{U_p}F|_p(V)$. 
The same argument as above gives the converse inclusion. It turns out that $(\mathfrak{a}_XF)|_p$ is a sheaf on $U_p$.
\end{proof}

\begin{prop}\label{prop:sheafification0}
The 
assignment $\mathfrak{a}_{X}$ from  $PSh_{\mathsf{Diff}}^{\mathcal{C}}\mathcal{D}_{X}$ to $Sh_{\mathsf{Diff}}^{\mathcal{C}}\mathcal{D}_X$ in Definition \ref{defn:sheafification} is a well-defined functor.  
\end{prop}
\begin{proof} 
Proposition \ref{prop:well-def.sheafification} implies that 
$\mathfrak{a}_XF$ is in $Sh_{\mathsf{Diff}}^{\mathcal{C}}\mathcal{D}_X$ for $F \in  PSh_{\mathsf{Diff}}^{\mathcal{C}}\mathcal{D}_X$. Then, 
it suffices to prove that 
the assignment $\mathfrak{a}_X$ is a functor. 
Let $\lambda : F \to G$ be a morphism of $PSh_{\mathsf{Diff}}^{\mathcal{C}}\mathcal{D}_X$, namely, a natural transformation. 
First, for $p \in \mathcal{D}_X$, we establish the well-definedness of $(\mathfrak{a}_X\lambda)_p:\mathfrak{a}_XF(p)\to\mathfrak{a}_XG(p)$ with a diagram 
\begin{eqnarray}\label{diag:widetilde.D.lambda}
\xymatrix@C25pt@R1pt{
&&&(F|_{p})_{v}\ar@{.>}[ddd]^(.4){\lambda_{v}}\\
F|_{p}(\mathcal{U}^{w})\ar[ddd]^(.4){\lambda_{\mathcal{U}^{w}}}\ar[rrru]^(.6){*}\ar[rr]^(.6){*}\ar[rd]^{*}&&(F|_{p})_{v'}\ar@{.>}[ddd]^(.4){\lambda_{v'}}&\\
&(F|_{p})_{v''}\ar@{.>}[ddd]^(.35){\lambda_{v''}}&&\\
&&&(G|_{p})_{v},\\
G|_{p}(\mathcal{U}^{w})\ar[rrru]^(.6){\star}\ar[rr]^(.7){\star}\ar[rd]^(.6){\star}&&(G|_{p})_{v'}&\\
&(G|_{p})_{v''}&&
}
\end{eqnarray}
where $w \in U_p$ and $\mathcal{U}^w$ is 
an open neighborhood of $w$ in $U_p$. By the condition $(**)$ in Definition \ref{defn:top.sheafification}, we can take an open subset $\mathcal{U}^w$ such that the components of  $(a_u)_{u\in U_p}$ with elements of $\mathcal{U}^w$ have a common section $s_{\mathcal{U}^w}\in F|_p(\mathcal{U}^w)$ satisfying $[s_{\mathcal{U}^w}]_v=a_v$ for every $v\in\mathcal{U}^w$. 
 The arrows with $*$ and $\star$ in (\ref{diag:widetilde.D.lambda}) form colimit diagrams of the stalks. By the definition, the map  $\lambda_{\mathcal{U}^w}$ is nothing but the restriction $\lambda_{p|_{\mathcal{U}^w}}$; that is,  for every morphism $\lambda$ of $PSh_{\mathsf{Diff}}^{\mathcal{C}}\mathcal{D}_X$, the family of maps $\{\lambda_p:F(p)\to G(p)\}_{p\in\mathcal{D}_X}$ can be seen as the collection of morphisms $\{\lambda_{U_p}:F|_p(U_p)\to G|_p(U_p)\}_{p\in\mathcal{D}_X}$ with  $\lambda_{U_p}\circ F(f)=G(f)\circ\lambda_{U_q}$ for every morphism $f:p\to q$ of $\mathcal{D}_X$.  The dotted arrows $\lambda_v, \lambda_{v'}$ and $\lambda_{v''}$ are induced by the universalities of $(F|_p)_v,(F|_p)_{v'}$ and  $(F|_p)_{v''}$, respectively. These constructions guarantee the commutativity of the diagram (\ref{diag:widetilde.D.lambda}). The arrows with $*$ send the section $s_{\mathcal{U}^w}$ to the  germ $[s_{\mathcal{U}^w}]_v=a_v\in(F|_p)_v$. For each $p$, we have 
 \begin{center}
 $(\mathfrak{a}_X\lambda)_p(([s_{\mathcal{U}^w}]_v)_{v\in U_p})=(\lambda_v([s_{\mathcal{U}^w}]_v))_{v\in U_p}=([\lambda_{\mathcal{U}^w}(s_{\mathcal{U}^w})]_v)_{v\in U_p}$.
 \end{center}
It is readily seen that the image is in $\mathfrak{a}_XG(p)$. 

To prove the naturality of $\mathfrak{a}_X\lambda$, 
it suffices to show that the diagram 
$$
\xymatrix@C30pt@R15pt{
\mathfrak{a}_XF(q)\ar[r]^{\mathfrak{a}_XF(f)}\ar[d]_{(\mathfrak{a}_X\lambda)_{q}}&\mathfrak{a}_XF(p)\ar[d]^{(\mathfrak{a}_X\lambda)_{p}}\\
\mathfrak{a}_XG(q)\ar[r]_{\mathfrak{a}_XG(f)}&\mathfrak{a}_XG(p)
}
$$
is commutative, where $f:p\to q$ is a morphism of $\mathcal{D}_{X}$. For a given $([s_{\mathcal{U}^w}]_v)_v\in \mathfrak{a}_XF(q)$, we have 
\begin{align*}
(\mathfrak{a}_X\lambda)_p\circ\mathfrak{a}_XF(f)(([s_{\mathcal{U}^w}]_v)_{v\in U_q})&=([\lambda_{f^{-1}(\mathcal{U}^w)}\circ F(f|_{f^{-1}(\mathcal{U}^w)})(s_{\mathcal{U}^w})]_u)_{u\in U_p} \ \ \text{and}\\
\mathfrak{a}_XG(f)\circ(\mathfrak{a}_X\lambda)_q(([s_{\mathcal{U}^w}]_v)_{v\in U_q})&=([G(f|_{f^{-1}(\mathcal{U}^w)})\circ\lambda_{\mathcal{U}^w}(s_{\mathcal{U}^w})]_u)_{u\in U_p}.
\end{align*}
The naturality of $\lambda$ enables us to conclude that the right-hand sides elements coincide with each other. Moreover, the uniqueness of each $\lambda_{v}$ yields the functoriality of $\mathfrak{a}_{X}$. Thus, the assignment $\mathfrak{a}_X$ is a functor to $Sh_{\mathsf{Diff}}^{\mathcal{C}}\mathcal{D}_X$.
\end{proof}

The functor $\mathfrak{a}_{X}$ is indeed the sheafification functor. 

\begin{prop}\label{prop:sheafification}
 The functor $\mathfrak{a}_{X}$ is the left adjoint to the inclusion functor $\mathfrak{i}_{X}:Sh_{\mathsf{Diff}}^{\mathcal{C}}\mathcal{D}_{X}\hookrightarrow PSh_{\mathsf{Diff}}^{\mathcal{C}}\mathcal{D}_{X}$ and possesses the universality of the sheafificaion functor.
 \end{prop}
 \begin{proof}
 See the proof in Appendix \ref{sect:AppB}.
 \end{proof}

\section{Constructions of the functors $\widetilde{C}$ and $\widetilde{D}$}\label{sect:CD}
We construct functors $\widetilde{C}:Sh_{\mathsf{Top}}^{\mathcal{C}}Open(X)\to Sh_{\mathsf{Diff}}^{\mathcal{C}}\mathcal{D}_{CX}$ and $\widetilde{D}:Sh_{\mathsf{Diff}}^{\mathcal{C}}\mathcal{D}_{X}\to Sh_{\mathsf{Top}}^{\mathcal{C}}Open(DX)$ 
by using $C:\mathsf{Top} \to \mathsf{Diff}$ and $D:\mathsf{Diff} \to \mathsf{Top}$; see Section  
\ref{sect:diff} for the adjunction $(D,C)$ . As mentioned in Introduction, 
there is no natural functor between 
the sites $\mathsf{Diff}$ and $\mathsf{Top}$. 
Therefore, we construct explicitly the functors $\widetilde{C}$ and $\widetilde{D}$ by making use of the sheafification functor, the limit and colimit functors.

\subsection{The construction of $\widetilde{C}$}
\begin{defn}\label{defn:overlineC}
Let $(X,Open(X))$ be a pair of topological space and its topology. We regard $Open(X)$ as the category whose morphisms are inclusions of open subsets of $X$. 
The functor $C:\mathsf{Top}\to\mathsf{Diff}$ yields the category $\mathcal{D}_{CX}$ of plots of $C(X,Open(X))$. Then, we define a functor 
$\overline{C}:Sh_{\mathsf{Top}}^{\mathcal{C}}Open(X)\to PSh_{\mathsf{Diff}}^{\mathcal{C}}\mathcal{D}_{CX}$ by 
$$
\overline{C}\theta_{X}(p)\hspace{5pt}=\hspace{5pt}p^{*}\theta_{X}(U_{p})\hspace{5pt}=\hspace{5pt}\text{colim}_{p(U_{p})\subset V, V\in Open(X)}\hspace{3pt}\theta_{X}(V)
$$
for $p\in \mathcal{D}_{CX}$ and $\theta_X\in Sh_{\mathsf{Top}}^{\mathcal{C}}Open(X)$. 
A morphism $\overline{C}\theta_X(f):\overline{C}\theta_X(q)\to\overline{C}\theta_X(p)$ for a  morphism $f:p\to q$ on $\mathcal{D}_{CX}$ is defined by the universality of $\overline{C}\theta_X(q)$.

For $\lambda:\theta_X\to\theta'_X$ in $Sh_{\mathsf{Top}}^{\mathcal{C}}Open(X)$, a morphism $\overline{C}\lambda:\overline{C}\theta_X\longrightarrow\overline{C}\theta'_X$ is defined 
by 
\begin{center}
$
(\overline{C}\lambda)_p([s_V]_{p(U_p)}):=[\lambda_V(s_V)]_{p(U_p)}
$
\end{center}
for $p \in \mathcal{D}_{CX}$, 
where $\lambda_V:\theta_X(V)\to\theta'_X(V)$ is the morphism induced by $\lambda$. An  element in the colimit is described as the square brackets with the subscript $p(U_p)$. 
\end{defn}

We verify the well-definedness of $\overline{C}$. 
\begin{prop}\label{prop:overlineC}
Let $(X,Open(X))$ be a topological space and $Open(X)$ the category of open subsets of $X$. Then, the functor $\overline{C}:Sh_{\mathsf{Top}}^{\mathcal{C}}Open(X)\to PSh_{\mathsf{Diff}}^{\mathcal{C}}\mathcal{D}_{CX}$ in Definition \ref{defn:overlineC} is well defined.
\end {prop}
\begin{proof}
Let $\theta_X$ be a sheaf in $Sh_{\mathsf{Top}}^{\mathcal{C}}Open(X)$ and $f:p\to q$ a morphism in $\mathcal{D}_{CX}$.  Since $p(U_p)\subset q(U_q)$, it follows that the unique map $f^*_{U_{q}}:\overline{C}\theta_X(q)\to\overline{C}\theta_X(p)$ is induced  by the universality of $\overline{C}\theta_X(q)$. Similarly, for any open sets $W\subset U_q$, we have a map $f^*_W:\overline{C}\theta_X(q|_W)=q^*\theta_X(W)\longrightarrow p^*\theta_X(f^{-1}(W))=\overline{C}\theta_X(p|_{f^{-1}(W)})$ defined by the universality of the colimit.
\[
\xymatrix@C15pt@R8pt{
&p^*\theta_X(f^{-1}(W))&&&X&\\
&&&&&\\
&q^*\theta_X(W)\ar@{.>}[uu]_{f^*_W}&& U_p\ar[rr]^f\ar[ruu]^p&&U_q\ar[luu]_q\\
&&&&&\\
\theta_X(V)\ar[rr]^{\theta_X(V'\hookrightarrow V)}\ar[ruuuu]\ar[ruu]&&\theta_X(V')\ar[luuuu]\ar[luu]&f^{-1}(W)\ar[rr]^-{f|_{f^{-1}(W)}}\ar@{_(->}[uu]&&W\ar@{_(->}[uu]
}
\]
Then, it is readily seen that $\overline{C}\theta_X(f'\circ f)=(f'\circ f)^*=f^*\circ f'^*=\overline{C}\theta_X(f)\circ\overline{C}\theta_X(f')$ for each composite of morphisms $p\xrightarrow{f}q\xrightarrow{f'}r$ and $\overline{C}\theta_X(id_p)=id_{\overline{C}\theta_X(p)}$.
As the same argument above, we see that  $\overline{C}\lambda:\overline{C}\theta_X\to\overline{C}\theta'_X$  satisfies the conditions of
 functor because for each $p$, the map $\overline{C}\lambda_p:\overline{C}\theta_X(p)\to\overline{C}\theta'_X(p)$ is also induced by the universality of $\overline{C}\theta_X(p)$.
\end{proof}

\begin{defn}\label{defn:widetildeC}
A functor $\widetilde{C}:Sh_{\mathsf{Top}}^{\mathcal{C}}Open(X)\to Sh_{\mathsf{Diff}}^{\mathcal{C}}\mathcal{D}_{CX}$ is defined by the composite 
$$
\widetilde{C}\theta_{X}:=\mathfrak{a}_{X}\overline{C}\theta_{X}
$$
of the sheafification functor  $\mathfrak{a}_X$ in Definition  \ref{defn:sheafification} 
and $\overline{C}$.
\end{defn}

\begin{prop}\label{well-def:widetildeC}
The functor $\widetilde{C}:Sh^{\mathcal{C}}_{\mathsf{Top}}Open(X)\to Sh_{\mathsf{Diff}}^{\mathcal{C}}\mathcal{D}_{CX}$ is well defined. 
\end{prop}

\begin{proof}
By the Definition \ref{defn:widetildeC}, Propositions \ref{prop:sheafification0}
 and \ref{prop:overlineC}, we see that  $\widetilde{C}$ is a functor taking values in  $Sh_\mathsf{Diff}^{\mathcal{C}}\mathcal{D}_{CX}$.
\end{proof}

\subsection{Construction of $\widetilde{D}$}

\indent Let $\mathcal{D}_{X}$ be the category of plots on diffeological space $(X,\mathcal{D}_X)$ and $Open(DX)$ the category consisting of $D$-open subsets of $D(X,\mathcal{D}_X)$ as objects and inclusions as morphisms.

\begin{defn}\label{defn:wildetildeD}
A functor $\widetilde{D}:Sh_{\mathsf{Diff}}^{\mathcal{C}}\mathcal{D}_X\to Sh_{\mathsf{Top}}^{\mathcal{C}}Open(DX)$ is defined. 
Let $(X,\mathcal{D}_X)$ be a diffeological space, $F$ an object of $Sh_{\mathsf{Diff}}^{\mathcal{C}}\mathcal{D}_X$ and $f:p\to q$ a morphism in $\mathcal{D}_X$. For any open subsets of $V\in Open(DX)$, we obtain the following commutative diagram 
$$
\xymatrix@C=30pt@R=10pt{
&&X&&\\
&U_{p}\ar[rr]^{f}\ar[ru]^{p}&&U_{q}\ar[lu]_{q}&\\
p^{-1}(V)\ar@{^{(}->}[ru]\ar@/^24pt/[rruu]^{p|_{p^{-1}(V)}}\ar[rrrr]^{f|_{p^{-1}(V)}}&&&&q^{-1}(V).\ar@{_{(}->}[lu]\ar@/_24pt/[lluu]_{q|_{q^{-1}(V)}}
}
$$
Then, it follows that $F(f|_{p^{-1}(V)}):F(q|_{q^-1(V)})=F|_{q}(q^{-1}(V))\longrightarrow F|_{p}(p^{-1}(V))=F(p|_{p^{-1}(V)})$. These arrows compose a family 
$$
\{F(f|_{p^{-1}(V)}):F|_{q}(q^{-1}(V))\longrightarrow F|_{p}(p^{-1}(V))\}_{p,q\in Ob(\mathcal{D}_{X}),f\in Mor(\mathcal{D}_{X})}.
$$
With the family, we define the functor $\widetilde{D}F:Open(DX)\to\mathcal{C}$ by 
\begin{center}
$\widetilde{D}F(V):=\text{lim}_{p\in \mathcal{D}_{X}}\hspace{2pt}F|_{p}(p^{-1}(V))$.
\end{center}
Moreover, for each morphism $\lambda:F\to F'$ of $Sh_{\mathsf{Diff}}^{\mathcal{C}}\mathcal{D}_{X}$ and $V\in D(X,\mathcal{D}_X)$, the universality of $\widetilde{D}F'(V)$ gives a morphism $\widetilde{D}\lambda_V:\widetilde{D}F(V)\to\widetilde{D}F'(V)$ defined by 
\[
\widetilde{D}\lambda_V((s_{pV})_{p\in\mathcal{D}_X})=(\lambda_{{p^{-1}(V)}}(s_{pV}))_{p\in\mathcal{D}_X}
\]
for each $(s_{pV}\in F|_p(p^{-1}(V)))_{p\in\mathcal{D}_X}\in\widetilde{D}F(V)$, where $\lambda_{p^{-1}(V)}$ is a map $F|_p(p^{-1}(V))\to G|_p(p^{-1}(V))$ that is equal to $\lambda_{p|_{p^{-1}(V)}}:F(p|_{p^{-1}(V)})\to G(p|_{p^{-1}(V)})$. The construction above  gives rise to the functor $\widetilde{D}:Sh_{\mathsf{Diff}}^{\mathcal{C}}\mathcal{D}_X\to Sh_{\mathsf{Top}}^{\mathcal{C}}Open(DX)$.
\end{defn}

\begin{rem}
The construction of $\widetilde{D}$ is same with one of the section of a functor described by Akbar Dehghan Nezhad and Alireza Ahmadi; see \cite[Definition 3.4]{AA}.
\end{rem}

\begin{rem}\label{rem.widetildeD}
For $F\in Sh_{\mathsf{Diff}}^{\mathcal{C}}\mathcal{D}_X$, the sheaf $F|_{p}$ over the domain of the plot $p$ sends $V_{i}\hookrightarrow V$ to the restriction map $F|_{p}(p^{-1}(V))\to F|_{p}(p^{-1}(V_{i}))$ and then the universality of the limit induces the unique arrow $\widetilde{D}F(V)\to\widetilde{D}F(V_{i})$. The limit construction enables us to conclude that $\widetilde{D}F$ is a functor, namely, a presheaf. 
In the same way, the functoriality of $\widetilde{D}$ is verified. 
\end{rem}

We need to verify that $\widetilde{D}F$ is a sheaf. 

\begin{prop}\label{prop:well-def of wildetildeD}
If $F$ is a $\mathcal{C}$-valued sheaf on $\mathcal{D}_{X}$, then $\widetilde{D}F$ is a $\mathcal{C}$-valued sheaf on $Open(DX)$.
\end{prop}
\begin{proof}
Let $W$ be in $Open(DX)$ and $\{W_{i}\hookrightarrow W\}_i$ a family of morphisms of $Open(DX)$ with $\cup_iW_i =W$. Since $F$ is a $\mathcal{C}$-valued sheaf on a diffeological space $(X, \mathcal{D}_X)$, we have an equalizer 
$$
\xymatrix{
F|_{q}(q^{-1}(W))\ar[r]^(.45){e^{q}}&\prod_{i}\hspace{2pt}F|_{q}(q^{-1}(W_{i}))\ar@<.9ex>[r]^(.45){\pi_{+}^{q}}\ar@<-0.9pt>[r]_(.45){\pi_{-}^{q}}&\prod_{i<j}\hspace{2pt}F|_{q}(q^{-1}(W_{i}\cap W_{j}))
}
$$
 in $\mathcal{C}$ for $q\in \mathcal{D}_{X}$, where $e^{q}$ is constructed by restriction map for $i$. The maps $\pi_{+}^{q} $ and $ \pi_{-}^{q}$ are induced by $W_{i}\cap W_{j}\hookrightarrow W_{i}\hspace{3pt}$ and $\hspace{3pt}W_{i}\cap W_{j}\hookrightarrow W_{j}$ for each $i<j$, respectively. For each $i\in I$ and a morphism  $f:p\to q $ of $\mathcal{D}_X$, we see that the diagram
$$
\xymatrix@C=45pt@R=20pt{
F|_{q}(q^{-1}(W))\ar[r]^{F(f|_{p^{-1}(W)})}\ar[d]&F|_{p}(p^{-1}(W))\ar[d]\\
F|_{q}(q^{-1}(W_{i}))\ar[r]^{F(f|_{p^{-1}(W_{i})})}&F|_{p}(p^{-1}(W_{i}))
}
$$
is commutative. Thus, we have the following commutative diagram $(*)$: 

{\small 
\[
\xymatrix@C=10pt@R=17pt{
\text{lim}_{q\in \mathcal{D}_{X}}F|_{q}(q^{-1}(W))\ar[r]^(.45){\text{lim}\hspace{2pt}e}\ar[d]&\text{lim}_{q\in \mathcal{D}_{X}}\prod_{i}\hspace{2pt}F|_{q}(q^{-1}(W_{i}))\ar@<.9ex>[r]^(.45){\text{lim}\pi_{+}}\ar@<-0.9pt>[r]_(.45){\text{lim}\pi_{-}}\ar[d]^{\star\star}&\text{lim}_{q\in \mathcal{D}_{X}}\prod_{i<j}\hspace{2pt}F|_{q}(q^{-1}(W_{i}\cap W_{j}))\ar[d]\\
F|_{q}(q^{-1}(W))\ar[r]^(.45){e^{q}}\ar[d]&\prod_{i}\hspace{2pt}F|_{q}(q^{-1}(W_{i}))\ar@<.9ex>[r]^(.45){\pi_{+}^{q}}\ar@<-0.9pt>[r]_(.45){\pi_{-}^{q}}\ar[d]^{\star}&\prod_{i<j}\hspace{2pt}F|_{q}(q^{-1}(W_{i}\cap W_{j}))\ar[d]\\
F|_{p}(p^{-1}(W))\ar[r]^(.45){e^{p}}\ar[d]&\prod_{i}\hspace{2pt}F|_{p}(p^{-1}(W_{i}))\ar@<.9ex>[r]^(.45){\pi_{+}^{p}}\ar@<-0.9pt>[r]_(.45){\pi_{-}^{p}}\ar[d]&\prod_{i<j}\hspace{2pt}F|_{p}(p^{-1}(W_{i}\cap W_{j}))\ar[d]\\
&&\hspace{20pt}.
}
\]
}
The second and subsequent rows are equalizers.  
On each column, the vertical arrows consist of limit diagrams. The maps $\text{lim} \hspace{2pt}\pi_{+}$, $\text{lim}\hspace{2pt}\pi_{-}$ and $\text{lim} \hspace{2pt}e$ are induced by the universalities of $\text{lim}_{q\in \mathcal{D}_{X}}\prod_{i<j}\hspace{2pt}F|_{q}(q^{-1}(W_{i}\cap W_{j}))$ and $\text{lim}_{q\in \mathcal{D}_{X}}\prod_{i}\hspace{2pt}F|_{q}(q^{-1}(W_{i}))$, respectively.

In order to varify that the top row is an equalizer diagram, let $A\in\mathcal{C}$ and $\chi:A\to \text{lim}_{q\in \mathcal{D}_{X}}\prod_{i}\hspace{2pt}F|_{q}(q^{-1}(W_{i}))$ be a morphism satisfying $\text{lim}\hspace{2pt}\pi_{+}\circ\chi=\text{lim} \hspace{2pt}\pi_{-}\circ\chi$. This equality yields that there is a map $\chi^{q}:A\to \prod_iF|_q(q^{-1}(W_i))$ such that $\pi_{+}^{q}\circ\chi^{q}=\pi_{-}^{q}\circ\chi^{q}$ for all $q\in\mathcal{D}_{X}$. Each $\chi^{q}$ forms the commutative triangle with the middle vertical arrows; that is, $\chi^q=\star\star\circ\chi$ and $\chi^p=\star\circ\chi^q$ in the diagram $(*)$ above. 
Therefore, by the universality of $F|_{q}(q^{-1}(W))$, there exists the unique arrow $\alpha^{q}:A\to F|_{q}(q^{-1}(W))$ such that  $\chi^{q}=e^{q}\circ\alpha^{q}$ for each $q\in\mathcal{D}_X$. 

The commutativity of the diagram $(*)$ allows us to obtain a family $\{\alpha^{q}\}_{q\in\mathcal{D}_{X}}$ of morphisms which satisfies the condition that 
 $F(f|_{p^{-1}(W)})\circ\alpha^{q}=\alpha^{p}$ for each $f:p\to q$. By the universality of $\text{lim}_{q\in \mathcal{D}_{X}}F|_{q}(q^{-1}(W))$, there is a unique arrow $\alpha:A\to \text{lim}_{q\in \mathcal{D}_{X}}F|_{q}(q^{-1}(W))$ satisfying $\chi=\text{lim}\hspace{2pt}e\hspace{2pt}\circ\alpha$. 
Finally, limit and product are commutative, so we have the result.
\end{proof}

\section{Proofs of main results}\label{sect:Proofs}

We prove Theorem A and Corollary B by considering composites of the limits, colimits and the shafification which define the functor $\widetilde{C}$ and 
$\widetilde{D}$. 
To this end, we introduce two functors $\varepsilon_{X*}$ and $\eta_{X}^{*}$ induced by counit and unit of the adjoint pair $C$ and $D$. Then, we have adjunctions $\widetilde{C}\dashv \varepsilon_{X*}\widetilde{D}$ and $\eta^{*}_{X}\widetilde{C}\dashv\widetilde{D}$. 

\subsection{The first main theorem}\label{sect:Thm_A_1}
We begin by recalling the unit and counit of the adjunction $D:\mathsf{Diff}\leftrightarrows\mathsf{Top}:C$ more precisely.
 Let $(X,\mathcal{D}_{X})$ be a diffeological space and $\eta$ the unit associated with the adjunction $(C, D)$. The smooth map $$\eta_{X}:=\eta_{(X,\mathcal{D}_{X})}:(X,\mathcal{D}_{X})\to CD(X,\mathcal{D}_{X})$$ 
 for a diffeological space $(X,\mathcal{D}_{X})$ is the identity map on the underlying set $X$.

 \begin{defn}\label{defn:eta^*}
 (The functor $\eta_{X}^{*}$) 
 Let $\mathcal{D}_X$ and $\mathcal{D}_{CDX}$ denote the categories of plots of $(X,\mathcal{D}_X)$ and $CD(X,\mathcal{D}_X)$,  respectively. Then, the map $\eta_X$ induces the functor $\eta^{*}_{X}:Sh_{\mathsf{Diff}}^{\mathcal{C}}\mathcal{D}_{CDX}\to Sh_{\mathsf{Diff}}^{\mathcal{C}}\mathcal{D}_{X}$ defined by
\begin{center}
$\eta_{X}^{*}F(p):=F(\eta_{X}\circ p),\hspace{5pt}(\eta_{X}^{*}\lambda)_{p}:=\lambda_{\eta_{X}\circ p}:F(\eta_{X}\circ p)\to G(\eta_{X}\circ p)$
\end{center}
for each $p\in\mathcal{D}_{X}$, $F\in Sh_{\mathsf{Diff}}^{\mathcal{C}}\mathcal{D}_{CDX}$ and morphism $\lambda:F\to G$ in $Sh_{\mathsf{Diff}}^{\mathcal{C}}\mathcal{D}_{CDX}$. It is easily seen that $\eta_{X}^{*}F$ is a sheaf. 
\end{defn}

The following is one of the main results in this manuscript.
\begin{thm}\label{thm:main,eta.ver.}
 \text{\rm (The adjunction $\eta_{X}^{*}\widetilde{C}\dashv\widetilde{D}$)} 
Let $\mathcal{D}_{X}$ be the category of plots on a diffeological space $(X,\mathcal{D}_X)$ and $\eta_{X}^{*}$ the functor defined in Definition \ref{defn:eta^*}. Then, the composite $\eta_{X}^{*}\widetilde{C}$ is the left adjoint functor to $\widetilde{D}$.
$$
\xymatrix@C=8pt@R=15pt{
Sh_{\mathsf{Diff}}^{\mathcal{C}}\mathcal{D}_{X}\ar[rr]^(.45){\widetilde{D}}&&Sh_{\mathsf{Top}}^{\mathcal{C}}Open(DX)\ar[ld]^{\widetilde{C}}\\
&Sh_{\mathsf{Diff}}^{\mathcal{C}}\mathcal{D}_{CDX}\ar[lu]^{\eta^{*}_{X}}\ar@{}[u]|{\top}&
}
$$
\end{thm}

In order to prove the theorem, 
we construct natural transformations $\alpha:\eta_{X}^{*}\widetilde{C}\widetilde{D}\to id_{Sh_{\mathsf{Diff}}^{\mathcal{C}}\mathcal{D}_{X}}$ and $\beta: id_{Sh_{\mathsf{Top}}^{\mathcal{C}}Open(DX)}\to \widetilde{D}(\eta_{X}^{*}\widetilde{C})$. Moreover, we verify whether $\eta_X^*\widetilde{C}$ and $\widetilde{D}$ together with $\alpha$ and $\beta$ satisfy the triangle identities; that is, the following triangle diagrams 
$$
 \xymatrix@C=35pt@R=30pt
 {
\eta_{X}^{*}\widetilde{C}F(p)\ar[r]^(.4){\eta^{*}_{X}\widetilde{C}\beta_{F}(p)}\ar[rd]_{id_{\eta^{*}_{X}\widetilde{C}F}(p)}&(\eta^{*}_{X}\widetilde{C})\widetilde{D}(\eta^{*}_{X}\widetilde{C})F(p)\ar[d]^{\alpha_{\eta^{*}_{X}\widetilde{C}F}(p)} &\widetilde{D}G(W)\ar[r]^(.4){\beta_{\widetilde{D}G}(W)}\ar[rd]_{id_{\widetilde{D}G}(W)}&\widetilde{D}(\eta^{*}_{X}\widetilde{C})\widetilde{D}G(W)\ar[d]^{\widetilde{D}\alpha_{G}(W)}\\
\ar@{}[ru]|(.65){\circlearrowright}&\eta^{*}_{X}\widetilde{C}F(p), &\ar@{}[ru]|(.65){\circlearrowright}&\widetilde{D}G(W)
}
$$
are commutative for $F\in Sh_{\mathsf{Top}}^{\mathcal{C}}Open(DX)$, $G\in Sh_{\mathsf{Diff}}^{\mathcal{C}}\mathcal{D}_{X}$, $p\in \mathcal{D}_{X}$ and $W\in Open(DX)$. 
\begin{const}\label{const:alpha}
 (The natural transformation $\alpha$) 
Fix $F\in Sh_{\mathsf{Diff}}^{\mathcal{C}}\mathcal{D}_{X}$ and $p\in\mathcal{D}_{X}$. According to the definitions of $\eta_X^*$, $\widetilde{C}$ and $\widetilde{D}$, we obtain
\begin{align*}
(\eta_{X}^{*}\widetilde{C}\widetilde{D})F(p)=\widetilde{C}\widetilde{D}F(\eta_{X}p)&= a_{U_{\eta_{X}p}}(\eta_{X}p)^{*}\widetilde{D}F(U_{\eta_{X}p})\\
&= \{(a_{u})_{u\in U_{\eta_{X} p}}\in\prod_{u\in U_{\eta_{X}p}}((\eta_{X}p)^{*}\widetilde{D}F)_{u}\hspace{3pt}|\hspace{3pt}(**)\hspace{3pt}\},
\end{align*}
where $(**)$ denotes the condition given by the definition of the sheafification functor $a_{U_{\eta_{X}p}}$; see the condition $(**)$ of Definition \ref{defn:top.sheafification}.
Here $((\eta_{X}p)^{*}\widetilde{D}F)_{v}$ is the stalk of $(\eta_{X}p)^{*}\widetilde{D}F$ at $u$; that is, $$((\eta_{X}p)^{*}\widetilde{D}F)_{u}=\text{colim}_{u\in W}\hspace{3pt}\text{colim}_{\eta_{X}p(W)\subset V}\hspace{3pt}\text{lim}_{q\in \mathcal{D}_X}F|_{q}(q^{-1}(V))
.$$
Moreover, the element $a_u$ denotes the germ at $u$ of an element in $(\eta_{X}p)^{*}\widetilde{D}F(U_{\eta_{X}p})$ for some domain $U_{\eta_{X}p}$ of the plot $\eta_Xp$. 
 Let $[-]_{\eta_Xp(W)}$ 
 and $[-]_v$ denote the classes in the first and second colimits, respectively. Then, the germ $a_{v}$ has an element $(s_{q}\in F|_{q}(q^{-1}(V)))_{q\in \mathcal{D}_{X}}\in\widetilde{D}F(V)$ as a representative. Thus, we have
\[
a_{v}=[[(s_{q})_{q\in \mathcal{D}_{X}}]_{\eta_{X}p(W)}]_v,
\]
\noindent where $(s_q)_{q\in\mathcal{D}_X}$ is an element in $\text{lim}_{q\in\mathcal{D}_X}F|_q(q^{-1}(V)).$ \\

In order to construct the natural transformation $\alpha$ mentioned above, we define a morphism $\alpha_{F(p)}:\eta_{X}^{*}\widetilde{C}\widetilde{D}F(p)\to F(p)$. 
We take an element $$([[\hspace{2pt}s_{p}\in F|_{p}(p^{-1}(V))]_{\eta_{X}p(W)}]_{u})_{u\in U_{\eta_Xp}}$$ indexed by $p$ from $([[(s_{q})_{q\in \mathcal{D}_{X}}]_{\eta_{X}p(W)}]_{u})_{u\in U_{\eta_Xp}}\in\eta_{X}^{*}\widetilde{C}\widetilde{D}F(p)$. Every component indexed by each element of $U_{\eta_Xp}$ satisfies the condition $(**)$. Therefore, for $w\in U_{\eta_{X} p}$, there exists an open neighborhood $\mathcal{U}^{w}$ and $\widetilde{\sigma}^{w}_{p}\in(\eta_{X} p)^{*}\widetilde{D}F(\mathcal{U}^{w})$ such that  $[\widetilde{\sigma}^{w}_{p}]_{v}=[[s_{p}\in F|_{p}(p^{-1}(V))]_{\eta_{X} p(W)}]_{v}$ for each $v\in\mathcal{U}^w$. Then, we have a $D$-open subset $\mathcal{V}$ satisfying $\eta_Xp(\mathcal{U}^w)\subset\mathcal{V}$ and $\sigma'^w_p\in F|_p(p^{-1}(\mathcal{V}))$ with $[\sigma'^w_p]_{\eta_Xp(\mathcal{U}^w)}=\widetilde{\sigma}_p^w$. This enables us to obtain a family $\{\sigma^{w}_{p}\in F|_{p}(\mathcal{U}^{w})\}_{w\in U_p}$, where $\sigma^{w}_{p}$ is defined by the restricted section $\sigma^{w}_{p}:=\sigma'^{w}_{p}|_{\mathcal{U}^{w}}$. We see that $ \mathcal{U}^{w}\subset p^{-1}\eta_X^{-1}(\mathcal{V})$ and moreover,  $\mathcal{U}^{w}\subset p^{-1}(\mathcal{V})$ because of $U_{\eta_Xp}=U_p$. Observe that $\eta_X$ is the identity map on the underlying set $X$. 

Since $F|_{p}$ is a sheaf on $U_p$, it follows that there exists a unique section $\Sigma_{p}\in F|_p(U_p)$ such that $\Sigma_{p}|_{\mathcal{U}^w}=\sigma^{w}_{p}$ for $w\in U_{\eta_{X} p}(=U_p)$. We define $\alpha_{F(p)}:\eta_{X}^{*}\widetilde{C}\widetilde{D}F(p)\to F(p)$ by
\begin{eqnarray}\label{def:alpha}
\alpha_{F(p)}(([[(s_{q})_{q\in \mathcal{D}_{X}}]_{\eta_{X} q(W)}]_{u})_{u\in U_{\eta_{X} p}})=\Sigma_{p}. 
\end{eqnarray}
We confirm the naturality of $\alpha$ in Lemma \ref{lem:well-def of alpha} .\\
\end{const}
In order to prove Theorem \ref{thm:main,eta.ver.}, we need lemmas.

\begin{lem}\label{lem:compatible}
The family $\{\sigma^{w}_{p}\in F|_{p}(\mathcal{U}^{w})\}_{w\in U_{\eta_{X}p}}$ defined in Construction \ref{const:alpha} is a compatible family. 
\end{lem}
\begin{proof}
Let $w $ and $ w'$ be elements of $U_{\eta_{X}p}$ satisfying $\mathcal{U}^{w}\cap\hspace{2pt}\mathcal{U}^{w'}\neq\phi$. To varify $\sigma_{p}^{w}|_{\mathcal{U}^{w}\cap\hspace{2pt}\mathcal{U}^{w'}}=\sigma_{p}^{w'}|_{\mathcal{U}^{w}\cap\hspace{2pt}\mathcal{U}^{w'}}$ for $w $ and $ w'$, we take an element $x\in\mathcal{U}^{w}\cap\hspace{2pt}\mathcal{U}^{w'}$.  By the construction of the family, we see that  $[\sigma_{p}^{w}]_{x}=a_{x}=[\sigma_{p}^{w'}]_{x}$, where $a_{x}$ is the same germ as in Construction \ref{const:alpha}. Then, there exists an open neighborhood $\mathcal{O}_{x}$ of $x$ such that  $\sigma_{p}^{w}|_{\mathcal{O}_{x}}=\sigma_{p}^{w'}|_{\mathcal{O}_{x}}$. We see  that the family $\{\mathcal{O}_{x}\}_{x}$ covers $\mathcal{U}^{w}\cap\hspace{2pt}\mathcal{U}^{w'}$. Thus, the equality  $\sigma_{p}^{w}|_{\mathcal{U}^{w}\cap\hspace{2pt}\mathcal{U}^{w'}}=\sigma_{p}^{w'}|_{\mathcal{U}^{w}\cap\hspace{2pt}\mathcal{U}^{w'}}$ is satisfied.
\end{proof}

\begin{lem}\label{lem:well-def of alpha}
The map $\alpha:\eta_{X}^{*}\widetilde{C}\widetilde{D}\longrightarrow id$ in Construction \ref{const:alpha} is a well-defined natural transformation.
\end{lem}
\begin{proof}
Let $F$ be an object of $Sh_{\mathsf{Diff}}^{\mathcal{C}}\mathcal{D}_{X}$ and $p$ a plot in  $\mathcal{D}_{X}$. Let $(a_{u})_{u\in U_{\eta_{X} p}}$ be an element of $\eta_{X}^{*}\widetilde{C}\widetilde{D}F(p)$. 
First, we check the well-definedness of $\alpha_{F(p)}$. Suppose that we obtain two different compatibile families $\{\sigma_{p}^{w}\in F|_{p}(\mathcal{U}^{w})\}_{w}$ and $\{\mu_{p}^{w}\in F|_{p}(\mathcal{V}^{w})\}_{w}$ from $(a_{u})_{u}$. The same argument as in the proof of Lemma \ref{lem:compatible} enables us to deduce that  $\sigma_{p}^{w}|_{\mathcal{U}^{w}\cap\hspace{2pt}\mathcal{V}^{w}}=\mu_{p}^{w}|_{\mathcal{U}^{w}\cap\hspace{2pt}\mathcal{V}^{w}}$. Then, $\{\sigma_{p}^{w}\in F|_{p}(\mathcal{U}^{w})\}_{w}\cup\{\mu_{p}^{w}\in F|_{p}(\mathcal{V}^{w})\}_{w}$ is a compatible family. Since $F|_{p}$ is a sheaf, it follows that this compatible family induces the unique section $\Sigma_{p}\in F|_{p}(U_{p})$ satisfying $\Sigma_{p}|_{\mathcal{U}^{w}}=\sigma_{p}^{w}$ and $\Sigma_{p}|_{\mathcal{V}^{w}}=\mu_{p}^{w}$. Therefore, the map $\alpha_{F(p)}$ is well defined.

We show that the diagram 
\begin{eqnarray}\label{diag:nat.alpha}
\xymatrix@C=30pt@R=20pt{
\eta_{X}^{*}\widetilde{C}\widetilde{D}F(p)\ar[r]^{\alpha_{F(p)}}\ar[d]_{\eta_{X}^{*}\widetilde{C}\widetilde{D}\lambda_{p}}&F(p)\ar[d]^{\lambda_{p}}\\
\eta_{X}^{*}\widetilde{C}\widetilde{D}G(p)\ar[r]^{\alpha_{G(p)}}&G(p)
}
\end{eqnarray}
is commutative, where $\lambda:F\to G$ is an arrow of $Sh_{\mathsf{Diff}}^{\mathcal{C}}\mathcal{D}_{X}$ and $p$ is a plot of $\mathcal{D}_{X}$.
The following diagram (\ref{diag:lambda}) explains the constructions of elements of $\eta_{X}^{*}\widetilde{C}\widetilde{D}F(p)$, $\eta_{X}^{*}\widetilde{C}\widetilde{D}G(p)$ and $\eta_{X}^{*}\widetilde{C}\widetilde{D}\lambda_{p}$. 
{\small 
\begin{eqnarray}\label{diag:lambda}
\xymatrix@R=5pt@C=4pt{
&&&&&&((\eta_{X} p)^{*}\widetilde{D}G)_{w}\\
&(\eta_{X} p)^{*}\widetilde{D}G(U_{\eta_{X} p})\ar[r]&(\eta_{X} p)^{*}\widetilde{D}G(W)\ar[r]&\cdots \ar[r]&(\eta_{X} p)^{*}\widetilde{D}G(\mathcal{U}^{w})\ar[rru]^{\star}\ar[rr]^{\star}\ar[rrd]_{\star}&&((\eta_{X} p)^{*}\widetilde{D}G)_{w'}\\
&&&&&&((\eta_{X} p)^{*}\widetilde{D}G)_{w''}\\
&&&&&&((\eta_{X} p)^{*}\widetilde{D}F)_{w}\ar@{.>}[uuu]\\
\widetilde{D}G(V)\ar[ruuu]&(\eta_{X} p)^{*}\widetilde{D}F(U_{\eta_{X} p})\ar[r]\ar@{.>}[uuu]^{(\eta_{X} p)^{*}\lambda_{U_{\eta_{X}p}}}&(\eta_{X} p)^{*}\widetilde{D}F(W)\ar[r]\ar@{.>}[uuu]^{(\eta_{X} p)^{*}\lambda_{W}}&\cdots\ar[r]&(\eta_{X} p)^{*}\widetilde{D}F(\mathcal{U}^{w})\ar@{.>}[uuu]^{(\eta_{X} p)^{*}\lambda_{\mathcal{U}^{w}}}\ar[rru]^{\star}\ar[rr]^{\star}\ar[rrd]_{\star}&&((\eta_{X} p)^{*}\widetilde{D}F)_{w'}\ar@<1.0ex>@{.>}[uuu]\\
&&&&&&((\eta_{X} p)^{*}\widetilde{D}F)_{w''}\ar@<2.0ex>@{.>}[uuu]\\
&&&&&\\
\widetilde{D}F(V)\ar[rrrruuu]\ar[rruuu]\ar[uuu]^{\widetilde{D}\lambda_{V}}\ar[rr]^{\widetilde{D}F_{X}(V'\hookrightarrow V)}\ar[ruuu] && \widetilde{D}F(V')\ar[luuu]\ar[r]\ar[uuu]&\cdots&&
}
\end{eqnarray}
}

\noindent
The triangle with $\widetilde{D}F(V)$, $\widetilde{D}F(V')$ and $(\eta_Xp)^*\widetilde{D}F(U_{\eta_Xp})$ and the parallelogram with $\widetilde{D}\lambda_V$ and $(\eta_{X} p)^{*}\lambda_{U_{\eta_{Xp}}}$ in the left-hand side shows how $(\eta_{X} p)^{*}\lambda_{U_{\eta_{Xp}}}$ is made. 
In fact, the morphism $(\eta_{X} p)^{*}\lambda_{U_{\eta_{Xp}}}$ is induced by the universality of $(\eta_{X} p)^{*}\widetilde{D}F(U_{\eta_{X} p})$ and makes the parallelogram commutative. Moreover, note that $\widetilde{D}\lambda_{V}$ is obtained by the universality of $\widetilde{D}G(V)$. Then, for $(s_{qV}\in F|_{q}(q^{-1}(V)))_{q\in\mathcal{D}_{X}}\in\widetilde{D}F(V)$, we see that 
\begin{align*}
(\eta_{X} p)^{*}\lambda_{U_{\eta_{X} p}}([(s_{qV})_{q\in\mathcal{D}_{X}}]_{\eta_{X}p(U_{\eta_{X}p})})&=[\widetilde{D}\lambda_{V}((s_{qV})_{q\in\mathcal{D}_{X}})]_{\eta_Xp(U_{\eta_Xp})}\\
&=[(\lambda_{q^{-1}(V)}(s_{qV}))_{q\in\mathcal{D}_{X}}]_{\eta_Xp(U_{\eta_Xp})}.
\end{align*}
Here, $[-]_{\eta_Xp(U_{\eta_Xp})}$ is an element of $(\eta_Xp)^*\widetilde{D}F(U_{\eta_Xp})$ or $(\eta_Xp)^*\widetilde{D}G(U_{\eta_Xp})$. Observe that 
$[(s_{qV})_{q\in\mathcal{D}_{X}}]_{\eta_{X}p(U_{\eta_{X}p})}\in (\eta_Xp)^*\widetilde{D}F(U_{\eta_Xp})$ and $[\widetilde{D}\lambda_V((s_{qV})_{q\in\mathcal{D}_X})]_{\eta_Xp(U_{\eta_Xp})}\in (\eta_Xp)^*\widetilde{D}G(U_{\eta_Xp})$.

 For every open subset $W\subset U_{\eta_{X}p}$, the map $(\eta_{X} p)^{*}\lambda_{W}$ is induced by the same argument as above. The horizontal arrows $(\eta_Xp)^*\widetilde{D}G(U_{\eta_Xp})\to\cdots\to (\eta_{X}p)^*\widetilde{D}G(\mathcal{U}^w)$ and $(\eta_Xp)^*\widetilde{D}F(U_{\eta_Xp})\to\cdots\to (\eta_{X}p)^*\widetilde{D}F(\mathcal{U}^w)$ are induced by the inclusions $\mathcal{U}^w\subset\cdots\subset W\subset U_{\eta_Xp}$. 
 The stalks of $(\eta_{X}p)^{*}\widetilde{D}G$ and $(\eta_{X} p)^{*}\widetilde{D}F$ at each element $w'$ of $\mathcal{U}^w$ are indicated by $((\eta_{X}p)^{*}\widetilde{D}G)_{w'}$ and $((\eta_{X} p)^{*}\widetilde{D}F)_{w'}$. The morphisms with $\star$ in the diagram (\ref{diag:lambda}) consist of the natural maps of the colimit diagrams of stalks and 
 explain the property of elements of $\eta_X^*\widetilde{C}\widetilde{D}F(p)$; that is, for each $w\in U_p$, each element $(a_u)_{u\in U_{\eta_Xp}}$ of $\eta_X^*\widetilde{C}\widetilde{D}F(p)$ has an open set $\mathcal{U}^w$ and a section $\sigma^w$ of $(\eta_Xp)^*\widetilde{D}F(\mathcal{U}^w)$ such that $[\sigma^w]_v=a_v$ for every $v\in\mathcal{U}^w$. The diagram  (\ref{diag:lambda}) shows that each arrow with $\star$ sends the section $\sigma^w$ to $[\sigma^w]_{w'}=a_{w'}$. 
 The universalities of $\{((\eta_Xp)^*\widetilde{D}F)_v\}_{v\in\mathcal{U}^w}$ give 
 the dot arrows $\{(\eta_Xp)^*\lambda_v:((\eta_{X}p)^*\widetilde{D}F)_v \to \cdot \! \cdot \cdot \! \to
 ((\eta_Xp)^*\widetilde{D}G)_v\}_{v\in\mathcal{U}^w}$. 
 By construction, for every $w'\in\mathcal{U}^w$, each right rectangle including $(\eta_{X}p)^{*}\lambda_{\mathcal{U}^{w}}$ and $(\eta_{X} p)^{*}\lambda_{w'}$ is commutative. Thus, we see that 
\begin{center}
$\eta_{X}^{*}\widetilde{C}\widetilde{D}\lambda_{p}=\prod_{w\in U_{p}}(\eta_{X} p)^{*}\lambda_{w}$.
\end{center}

We are ready to show the commutativity of the diagram (\ref{diag:nat.alpha}). 
Let $\xi$ be an element of the form  $([[(s_{qV})_{q\in \mathcal{D}_{X}}]_{\eta_{X} p(\mathcal{U}^u)}]_{u})_{u\in U_{\eta_{X}p}}$ in $\eta_{X}^{*}\widetilde{C}\widetilde{D}F(p)$. Let $\Sigma_{p}$ be the image of $\xi$
by $\alpha_{F(p)}$ with a compatible family $\{\sigma_{p}^{w}\in F|_{p}(\mathcal{U}^{w})\}_{w\in U_{\eta_Xp}}$, where $w\in\mathcal{U}^w$, $p(\mathcal{U}^w)\subset \mathcal{V}^w$, $\sigma'^w_{p}\in F|_p(p^{-1}(\mathcal{V}^w))$ and $\sigma_{p}^{w}:=\sigma'^w_{p}|_{\mathcal{U}^{w}}$. 
In the diagram (\ref{diag:nat.alpha}),  we have   
$(\lambda_p\circ\alpha_{F(p)})
(\xi)
=\lambda_{p}(\Sigma_{p})$. 
Moreover, we see that 
\begin{align*}
(\eta_{X}^{*}\widetilde{C}\widetilde{D}\lambda_{p})(\xi)&=(\eta_{X}^{*}\widetilde{C}\widetilde{D}\lambda_{p})(([[(s_{qV})_{q\in \mathcal{D}_{X}}]_{\eta_{X}p(\mathcal{U}^u)}]_{u})_{u\in U_{\eta_{X} p}})\\
&=(\prod_{w\in U_{\eta_{X} p}}(\eta_{X} p)^{*}\lambda_{w})(([[(s_{qV})_{q\in \mathcal{D}_{X}}]_{\eta_{X} p(\mathcal{U}^u)}]_{u})_{u\in U_{\eta_{X} p}})\\
&=([\hspace{5pt}(\eta_{X}p)^{*}\lambda_{\mathcal{U}^u}([(s_{qV})_{q\in \mathcal{D}_{X}}]_{\eta_{X}p(\mathcal{U}^{u})})\hspace{5pt}]_{u})_{u\in U_{\eta_{X} p}}\\
&=([[(\lambda_{q|_{q^{-1}(V)}}(s_{qV}))_{q\in \mathcal{D}_{X}}]_{\eta_{X} p(\mathcal{U}^{u})}]_{u})_{u\in U_{\eta_{X} p}}.
\end{align*}
The second equality follows from the definition of $\eta_X^*\widetilde{C}\widetilde{D}\lambda_p$. The third is based on the commutative squares of the arrows with $\star$ and $(\eta_Xp)^*\lambda_{\mathcal{U}^w}$ in the diagram (\ref{diag:lambda}). The commutativity of the diagram with $\widetilde{D}\lambda_V$ and $(\eta_Xp)^*\lambda_{\mathcal{U}^w}$ gives the last equality.  

To calculate $\alpha_{G(p)}$, we take a compatible family $\{\lambda_{p|_{\mathcal{U}^{w}}}(\sigma_{p}^{w})\in G|_{p}(\mathcal{U}^{w})\}_{w}$ from $([[(\lambda_{q|_{q^{-1}(V)}}(s_{qV}))_{q\in \mathcal{D}_{X}}]_{\eta_{X} p(\mathcal{U}^{u})}]_{u})_{u\in U_{\eta_{X} p}}
$.
Let $\Sigma'_{p}$ be the unique section induced from the compatible family. Then, we have 
$$
\alpha_{G(p)}\circ(\eta^{*}_{X}\widetilde{C}\widetilde{D}\lambda_{p})(([[(s_{qV})_{q\in \mathcal{D}_{X}}]_{\eta_{X}p(W)}]_{u})_{u\in U_{\eta_{X} p}})=\Sigma'_{p},
$$  where it satisfies $\hspace{2pt}\Sigma_{p}'|_{\mathcal{U}^{w}}=\lambda_{p|_{\mathcal{U}^{w}}}(\sigma_{p}^{w})$ for all $ w\in U_{p}$.
 Note that the argument of the beginning of the proof 
 guarantees that images of $\alpha$ do not depend on the echoice of the compatible family which defines $\alpha$. 

The calculation of $\lambda_p\circ\alpha_{F(p)}$ yields the equalities $(\lambda_{p}\Sigma_{p})|_{\mathcal{U}^{w}}=\lambda_{p|_{\mathcal{U}^{w}}}(\Sigma_{p}|_{\mathcal{U}^w})=\lambda_{p|_{\mathcal{U}^{w}}}(\sigma_{p}^{w})$.
By the sheaf property (the uniqueness of extensions), we have the commutativity of the diagram (\ref{diag:nat.alpha}).
\end{proof}

\begin{const}\label{cons:beta}
 (The natural transformation $\beta$) 
 Let $(X,\mathcal{D}_X)$ be a diffeological space and $Open(DX)$ the category of $D$-open subsets of $D(X,\mathcal{D}_X)$. For $F\in Sh_{\mathsf{Top}}^{\mathcal{C}}Open(DX)$ and $W\in D(X,\mathcal{D}_X)$, we construct a morphism 
$
\beta_{F(W)}:F(W)\to\widetilde{D}\eta^{*}_{X}\widetilde{C}F(W)
$. By definition, we see that 
\begin{align*}
\widetilde{D}\eta^{*}_{X}\widetilde{C}F(W)\hspace{3pt}
&=\text{lim}_{q\in\mathcal{D}_{X}}\eta^{*}_{X}\widetilde{C}F|_{q}(q^{-1}(W))\hspace{3pt}=\hspace{3pt}\text{lim}_{q\in\mathcal{D}_{X}}\widetilde{C}F|_{\eta_Xq}((\eta_{X} q)^{-1}(W))\hspace{3pt}\\
&=\text{lim}_{q\in\mathcal{D}_{X}}\hspace{3pt}a_{U_{\eta_{X} q}}\hspace{3pt}(\eta_{X} q)^{*}F\hspace{3pt}((\eta_{X} q)^{-1}(W))\\
&= \text{lim}_{q\in\mathcal{D}_{X}}\hspace{3pt}\{(s_{u})_{u}\in\prod_{u\in (\eta_Xq)^{-1}(W)}((\eta_{X} q)^{*}F)_{u}|(**)\},
\end{align*}
where $(**)$ denotes the condition described in Definition  \ref{defn:top.sheafification}.

We assign an element of $\widetilde{D}\eta^*_X\widetilde{C}F(W)$ to 
an element 
$s_W$ in $F(W)$ as follows. Since $(\eta_{X} q)((\eta_{X}q)^{-1}(W))\subset W$, it follows that $[s_{W}]_{(\eta_{X} q)((\eta_{X} q)^{-1}(W))}$ is in 
$$(\eta_Xq)^*F((\eta_Xq)^{-1}(W))=\text{colim}_{(\eta_{X} q)((\eta_{X} q)^{-1}(W))\subset V} F((\eta_{X} q)^{-1}(V)).$$ 
A section $s_{W}$ gives 
an element of the form 
$([[s_{W}]_{(\eta_{X} q)((\eta_{X} q)^{-1}(W))}]_{u})_{u\in (\eta_Xq)^{-1}(W)}$
in $(\eta_{X} q)^{*}F((\eta_{X} q)^{-1}(W))$. It is immediate that this element satisfies the sheafification condition $(**)$. Furthermore, this construction does not depend on the choice of $q\in\mathcal{D}_X$, so we obtain the family 
$\{([[s_{W}]_{(\eta_{X} q)((\eta_{X}q)^{-1}(W))}]_{u})_{u\in (\eta_Xq)^{-1}(W)}\}_{q\in\mathcal{D}_{X}}$. Then define $\beta_{F(W)}$ by
\begin{eqnarray}\label{def:beta}
\beta_{F(W)}(s_{W})=(([[s_{W}]_{(\eta_{X} q)((\eta_{X}q)^{-1}(W))}]_{u})_{u\in (\eta_Xq)^{-1}(W)})_{q\in\mathcal{D}_{X}}. 
\end{eqnarray}
\end{const}

\begin{lem}
For each $F\in Sh_{\mathsf{Top}}^{\mathcal{C}}Open(DX)$ and $W\in D(X,\mathcal{D}_X)$, the map $\beta_{F(W)}$ in Construction \ref{cons:beta} is well defined.
\end{lem}
\begin{proof}
Because of the limit construction of $\widetilde{D}$, we need to check whether $\eta^{*}_{X}\widetilde{C}F(f)$ for $f:p\to q$ sends each $([[s_{W}]_{(\eta_{X} q)(\eta_{X} q)^{-1}(W)}]_{u})_{u\in (\eta_Xq)^{-1}(W)}$ to the element $([[s_{W}]_{(\eta_{X} p)(\eta_{X}p)^{-1}(W)}]_{u})_{u\in (\eta_Xp)^{-1}(W)}$ in Construction \ref{cons:beta}.
Observe that $\eta^*_X\widetilde{C}F(p)$ is a subset of $\prod_{u\in U_{\eta_Xp}}((\eta_Xp)^*F)_u$. The construction of $\eta^{*}_{X}\widetilde{C}F(f):\eta^*_X\widetilde{C}F(q)\to \eta^*_X\widetilde{C}F(p)$ gives the following commutative diagram:
\begin{eqnarray}
\xymatrix@C=5pt@R=8pt{
 && & & &((\eta_{X} p)^{*}F)_{v}\\
 &&(\eta_{X} p)^{*}F((\eta_{X} p)^{-1}(W))\ar[rrru]^{\star}\ar[rd]_{\star}\ar[rr]_{\star} & &((\eta_{X} p)^{*}F)_{v'} &\\
 && &((\eta_{X} p)^{*}F)_{v''} & &((\eta_{X} q)^{*}F)_{u}\ar@{.>}[uu]_{f_u}\\
 &&(\eta_{X} q)^{*}F((\eta_{X} q)^{-1}(W))\ar[uu]_(.3){(f|_{(\eta_{X} p)^{-1}(W)})^{*}}\ar[rd]_{\diamond}\ar[rr]_{\diamond}\ar[rrru]_{\diamond}& &((\eta_{X} q)^{*}F)_{u'}\ar@{.>}[uu]_{f_{u'}}&\\
 && &((\eta_{X} q)^{*}F)_{u''}\ar@{.>}[uu]_{f_{u''}}&&\\
 F(W)\ar[rruu]_{\Diamond}\ar[rruuuu]^(.7){\bigstar}&&&&&
}
\end{eqnarray}
Here, the large star and the small stars denote the assignments $s_W\mapsto [s_{W}]_{(\eta_{X} p)(\eta_{X}p)^{-1}(W)}$ and $[s_{W}]_{(\eta_{X} p)(\eta_{X}p)^{-1}(W)}\mapsto ([[s_{W}]_{(\eta_{X} p)(\eta_{X}p)^{-1}(W)}]_{u})_{u\in (\eta_Xp)^{-1}(W)}$, respectively. On the other hand, the large rhombus and the small rhombuses 
are defined by  $s_W\mapsto [s_{W}]_{(\eta_{X} q)(\eta_{X} q)^{-1}(W)}$ and $[s_{W}]_{(\eta_{X} q)(\eta_{X} q)^{-1}(W)}\mapsto([[s_{W}]_{(\eta_{X} q)(\eta_{X} q)^{-1}(W)}]_{u})_{u\in (\eta_Xq)^{-1}(W)}$. 
The map $(f|_{(\eta_Xp)^{-1}(W)})^*$ is induced by the universality of $(\eta_Xq)^*F((\eta_Xq)^{-1}(W))$; see Proposition \ref{prop:overlineC}. Note the relation of $f^{-1}((\eta_{X} q)^{-1}(W))=f^{-1} q^{-1}\eta^{-1}_{X}(W)=$ 
$p^{-1}\eta_{X}^{-1}(W)=(\eta_{X}p)^{-1}(W)$. Each dotted arrow $((\eta_{X} q)^{*}F)_{v}\dashrightarrow((\eta_{X} p)^{*}F)_{u}$ shows the map between stalks 
  whose subscripts satisfy $f(u)=v$, $f(u')=v'$ and $f(u'')=v''$.  Every square with a small star, a small rhombus 
  and $(f|_{(\eta_Xp)^{-1}(W)})^*$ is commutative because each dotted arrow is induced by the universality of the stalk. 
By definition, the map $\eta^*_X\widetilde{C}F(f)=\prod_{u\in U_{\eta_Xp}}f_u$ sends $([[s_{W}]_{(\eta_{X}q)(\eta_{X} q)^{-1}(W)}]_{v})_{v\in U_{\eta_{X} q}} $ to $ ([[s_{W}]_{(\eta_{X} p)(\eta_{X} p)^{-1}(W)}]_{u})_{u\in U_{\eta_{X} p}}$.
\end{proof}

The compatibility of $W\in D(X,\mathcal{D}_{X})$ is guaranteed by the construction of $\widetilde{C},\widetilde{D}$ and $\eta^{*}_{X}$. We show the naturality of $\{\beta_{F(W)}\}_{W}$ for $F\in Sh_{\mathsf{Top}}^{\mathcal{C}}Open(DX)$.

\begin{lem}\label{lem:well-def of beta}
Let $\{\beta_{F(W)}:F(W)\to \widetilde{D}\eta^{*}_{X}\widetilde{C}F(W)\}_{F\in Sh_{\mathsf{Top}}^{\mathcal{C}}Open(DX),W\in D(X,\mathcal{D}_{X})}$ be the family described in Construction \ref{cons:beta}. The family of maps is natural for each object $F$ in $Sh_{\mathsf{Top}}^{\mathcal{C}}Open(DX)$ and gives rise to the natural transformation $\beta:id_{Sh_{\mathsf{Top}}^{\mathcal{C}}Open(DX)}\to\widetilde{D}\eta^{*}_{X}\widetilde{C}$. 
\end{lem}
\begin{proof}
We need to verify the diagram
\begin{eqnarray}\label{diag:nat.beta}
\xymatrix@C=45pt{
F'(W)\ar[r]^(.45){\beta_{F'(W)}}&\widetilde{D}\eta^{*}_X\widetilde{C}F'(W)\\
F(W)\ar[r]^(.45){\beta_{F(W)}}\ar[u]^{\lambda_{W}}&\widetilde{D}\eta^{*}_{X}\widetilde{C}F(W)\ar[u]_{(\widetilde{D}\eta^{*}_X\widetilde{C}\lambda)_{W}}
}
\end{eqnarray}
is commutative, where $\lambda:F\to F'$ is a morphism of $Sh_{\mathsf{Top}}^{\mathcal{C}}Open(DX)$. 
We consider the diagram (\ref{diag:eta.c.lambda}) below that reveals the construction of $(\eta_X^*\widetilde{C}\lambda)_{q|_U}:\eta_X^*\widetilde{C}F(q|_{U})\to\eta_X^*\widetilde{C}F'(q|_U)$, where $U:=(\eta_Xq)^{-1}(W)$. 
\begin{eqnarray}\label{diag:eta.c.lambda}
\xymatrix@C=8pt@R=10pt{
 && & & &((\eta_{X} q)^{*}F')_{u}\\
 &&(\eta_{X} q)^{*}F'(U)\ar[rrru]\ar[rd]\ar[rr] & &((\eta_{X} q)^{*}F')_{u'} &\\
 && &((\eta_{X} q)^{*}F')_{u''} & &((\eta_{X} q)^{*}F)_{u}\ar@{.>}[uu]_{_{(\eta_{X} q)^{*}\lambda_{u}}}\\
 F'(W)\ar[rruu]&&(\eta_{X} q)^{*}F(U)\ar[uu]_(.3){_{(\eta_{X}q)^{*}\lambda_{U}}}\ar[rd]\ar[rr]\ar[rrru]& &((\eta_{X}q)^{*}F)_{u'}\ar@{.>}[uu]_(.6){_{(\eta_{X} q)^{*}\lambda_{u'}}}&\\
 && &((\eta_{X}q)^{*}F)_{u''}\ar@{.>}[uu]_(.25){_{(\eta_{X} q)^{*}\lambda_{u''}}}&&\\
 F(W)\ar[uu]^{\lambda_{W}}\ar[rruu]&&&&&
}
\end{eqnarray}
We observe that the maps $(\eta_{X} q)^{*}\lambda_{U}$ and $\{(\eta_{X} q)^{*}\lambda_{u}\}_{u\in U}$ are induced by the universalities of $(\eta_{X} q)^{*}F(U)$ and 
$\{((\eta_{X} q)^{*}F)_{u}\}_{u\in U}$, respectively. Then, this diagram (\ref{diag:eta.c.lambda}) is commutative. The map $(\eta_X^*\widetilde{C}\lambda)_{q|_U}$ is then defined by $(\eta_X^*\widetilde{C}\lambda)_{q|_U} =\prod_{u\in U}(\eta_Xq)^*\lambda_u$. 

 Let $s_W$ be an element of $F(W)$. By the commutativity of the diagram (\ref{diag:eta.c.lambda}), we have elements $[s_W]_{\eta_Xq(U)}\in (\eta_Xq)^*F(U)$, $\lambda_W(s_W)\in F'(W)$, $[\lambda_W(s_W)]_{\eta_Xq(U)}\in(\eta_Xq)^*F'(U)$ and $[[s_W]_{\eta_Xq(U)}]_u\in((\eta_Xq)^*F)_u$. It follows that
\begin{eqnarray*}
(\eta_{X} q)^{*}\lambda_{U}([s_{W}]_{\eta_{X}q(U)})&=&[\lambda_{W}(s_{W})]_{\eta_{X} q(U)}\hspace{10pt}\text{and}\hspace{5pt}\\
(\eta_{X}q)^{*}\lambda_{u}([[s_{W}]_{\eta_{X} q(U)}]_{u})&=&[(\eta_{X} q)^{*}\lambda_{U}([s_{W}]_{\eta_{X}q(U)})]_{u}\\
&=&[[\lambda_{W}(s_{W})]_{\eta_{X} q(U)}]_u.
\end{eqnarray*}
Moreover, we have
\begin{eqnarray*}
\widetilde{D}\eta^{*}_{X}\widetilde{C}\lambda_{W}\circ\beta_{F(W)}(s_W)&=&\widetilde{D}\eta^{*}_{X}\widetilde{C}\lambda_{W}((([[s_{W}]_{\eta_{X} q(U)}]_{u})_{u\in U})_{q\in\mathcal{D}_{X}})\\
&=&
(\hspace{2pt}(\eta_{X}^{*}\widetilde{C}\lambda)_{q|_U}((([[s_{W}]_{\eta_{X} q(U)}]_{u})_{u\in U})\hspace{2pt})_{q\in\mathcal{D}_{X}}\\
&=&(\hspace{2pt}(\hspace{2pt}(\eta_{X}q)^{*}\lambda_{u}([[s_{W}]_{\eta_{X} q(U)}]_{u})\hspace{2pt})_{u\in U})_{q\in\mathcal{D}_{X}}\\
&=&(\hspace{2pt}([[\lambda_{W}(s_{W})]_{\eta_{X} q(U)}]_{u})_{u\in U})_{q\in\mathcal{D}_{X}}=\beta_{F'(W)}\circ\lambda_W(s_W)
\end{eqnarray*}
\end{proof}


\begin{proof}[Proof of Theorem \ref{thm:main,eta.ver.}.] 
Let $F$ be an element of $Sh_{\mathsf{Top}}^{\mathcal{C}}Open(DX)$ and $p$  a plot of $\mathcal{D}_{X}$. Let $G$ be an object in $Sh_{\mathsf{Diff}}^{\mathcal{C}}\mathcal{D}_{X}$ and $W$ an open subset of $D(X,\mathcal{D}_{X})$. 
We verify that the diagrams 
$$
 \xymatrix@C=30pt@R=20pt
 {
\eta_{X}^{*}\widetilde{C}F(p)\ar[r]^(.35){\eta^{*}_{X}\widetilde{C}\beta_{F}(p)}\ar[rd]_{id_{\eta^{*}_{X}\widetilde{C}F}(p)}&(\eta^{*}_{X}\widetilde{C})\widetilde{D}(\eta^{*}_{X}\widetilde{C})F(p)\ar[d]^{\alpha_{\eta^{*}_{X}\widetilde{C}F}(p)}\ \ \ and \  & \hspace{-1cm}\widetilde{D}G(W)\ar[r]^(.3){\beta_{\widetilde{D}G}(W)}\ar[rd]_{id_{\widetilde{D}G}(W)}&\widetilde{D}(\eta^{*}_{X}\widetilde{C})\widetilde{D}G(W)\ar[d]^{\widetilde{D}\alpha_{G}(W)}\\
&\eta^{*}_{X}\widetilde{C}F(p) &&\widetilde{D}G(W)
}
$$
 are commutative.\\
 \indent First, we consider the left triangle. 
 Let  
 $([[s_{V}]_{\eta_{X} p(W)}]_{u})_{u\in U_{\eta_{X} p}}$ be an element 
 in $\eta^{*}_{X}\widetilde{C}F(p)$, where $s_V\in F(V)$ for $V\in D(X,\mathcal{D}_{X})$.  By the sheafification condition $(**)$ in Definition \ref{defn:top.sheafification}, 
 for each $w$, there exist a neighborhood $w\in\mathcal{U}^{w}$, an open subset $\mathcal{V}^w\in Open(DX)$ satisfying $p(\mathcal{U}^w)\subset\mathcal{V}^w$ and a section $s_{\mathcal{V}^w}\in F(\mathcal{V}^w)$  such that $[[s_{\mathcal{V}^w}]_{\eta_Xp(\mathcal{U}^w)}]_v=[[s_{V}]_{\eta_{X} p(W)}]_v$ for all $v\in \mathcal{U}^w$. 
 The element is mapped to the following element in $(\eta^{*}_{X}\widetilde{C})\widetilde{D}(\eta^{*}_{X}\widetilde{C})F(p)$ by $\eta^{*}_{X}\widetilde{C}\beta_{F(p)}$;
 
 \medskip
\begin{center}
$([[\bm{s_{\mathcal{V}^w}}]_{\eta_{X} p(\mathcal{U}^w)}]_{w})_{w\in U_{\eta_{X} p}}\longmapsto\Biggr(\biggr [\Bigr[\bm{\big(([[s_{\mathcal{V}^w}]_{\eta_{X} q(\eta_{X}q)^{-1}(\mathcal{V}^w)}]_{w})_{w\in(\eta_{X} q)^{-1}(\mathcal{V}^w)}\bigl)_{q\in\mathcal {D}_{X}}}\Bigl]_{\eta_{X} p(\mathcal{U}^w)}\biggl]_{w}\Bigg)_{w\in U_{\eta_{X} p}}$.
\end{center} 
Then, we apply $\alpha_{\eta^{*}_{X}\widetilde{C}F}(p)$ to this image. Recall that the first step of $\alpha$ is the projection $\pi_p$ to a component with the plot $p$. 
$$
\xymatrix@C=30pt@R=15pt{
\Biggr(\biggr [\Bigr[\big(([[s_{\mathcal{V}^w}]_{\eta_{X} q(\eta_{X} q)^{-1}(\mathcal{V}^w)}]_{v})_{v\in(\eta_{X} q)^{-1}(\mathcal{V}^w)}\bigl)_{q\in \mathcal{D}_{X}}\Bigl]_{\eta_{X} p(\mathcal{U}^w)}\biggl]_{w}\Bigg)_{w\in U_{\eta_{X} p}}\ar@{|->}[d]^{\pi_{p}}\\
\Biggr(\biggr [\Bigr[([[s_{\mathcal{V}^w}]_{\eta_{X} p(\eta_{X} p)^{-1}(\mathcal{V}^w)}]_{v})_{v\in(\eta_{X} p)^{-1}(\mathcal{V}^w)}\Bigl]_{\eta_{X}p(\mathcal{U}^w)}\biggl]_{w}\Bigg)_{w\in U_{\eta_{X} p}}
}
$$

The next step gives a compatible family. We define a section $\widetilde{\sigma}_{p}^{w}$ of  
$(\eta_Xp)^*\eta_X^*\widetilde{C}F(\mathcal{U}^w
)$ 
by
\begin{center}
$\hspace{2pt}\widetilde{\sigma}_{p}^{w}:=\Bigr[([[s_{\mathcal{V}^w}]_{\eta_{X} p(\eta_{X} p)^{-1}(\mathcal{V}^w)}]_{w})_{w\in(\eta_{X} p)^{-1}(\mathcal{V}^w)}\Bigl]_{\eta_{X} p(\mathcal{U}^w)}$. 
\end{center}
For every $v\in \mathcal{U}^w$, each $[\widetilde{\sigma}_{p}^{w}]_{v}$ (germ at $v$) is a component of the image of $\pi_p$ because $\beta$ is made of just the quotient maps at each $v$. We observe that $\sigma'^w_p:=([[s_{\mathcal{V}^w}]_{\eta_{X} p(\eta_{X} p)^{-1}(\mathcal{V}^w)}]_{v})_{v\in(\eta_{X} p)^{-1}(\mathcal{V}^w)}$ is a representative of $\widetilde{\sigma}^w_p$. We define $\sigma_{p}^{w}:=([[s_{\mathcal{V}^w}]_{\eta_{X} p(\eta_{X} p)^{-1}(\mathcal{V}^w)}]_{v})_{v\in \mathcal{U}^w}=([[s_{\mathcal{V}^w}]_{\eta_{X} p(\mathcal{U}^w)}]_{v})_{v\in \mathcal{U}^w}$ and obtain the compatible family $\{\sigma_{p}^{w}\in\eta^{*}_{X}\widetilde{C}F(p|_{\mathcal{U}^w})\}_{w\in U_{\eta_{X}p}}$. From the sheaf property of $\eta^{*}_{X}\widetilde{C}$, there exists a unique element $\Sigma_p$ of $\eta^{*}_{X}\widetilde{C}F(p)$ whose restriction to $\mathcal{U}^w$ is equal to $\sigma_{p}^{w}$. Observe that it does not depend on the choice of $\mathcal{U}^w$ as having seen in Lemma \ref{lem:well-def of alpha}. 
We have
\begin{center}
$\alpha_{\eta_{X}^{*}\widetilde{C}F(p)}\circ\eta_{X}^{*}\widetilde{C}\beta_{F(p)}(([[s_{V}]_{\eta_{X}p(W)}]_{u})_{u\in U_{\eta_{X}p}})=\Sigma_{p}$.
\end{center}
On the other hand, every restriction of the given element to $\mathcal{U}^w$ coincides with $\sigma_{p}^{w}$; that is,
$([[s_{V}]_{\eta_{X}p(W)}]_{u})_{u\in U_{\eta_{X} p}}|_{\mathcal{U}^w}=([[s_{\mathcal{V}^w}]_{\eta_{X} p(\mathcal{U}^w)}]_{v})_{v\in \mathcal{U}^w}=\sigma_{p}^{w}$.
By the sheaf property(the uniqueness), the section $\Sigma_{p}$ is equal to $([[s_{V}]_{\eta_X p(W)}]_{v})_{v\in U_{\eta_Xp}}$.

Next, we consider the right-hand side triangle. Recall that 
$G$ and $W$ are an object of $Sh_{\mathsf{Diff}}^{\mathcal{C}}\mathcal{D}_{X}$ and an open subset of $D(X,\mathcal{D}_{X})$, respectively. We take the element 
\begin{center}
$\zeta:= (s_{qW}\in G|_{q}(q^{-1}(W)))_{q\in\mathcal{D}_{X}}\in \hspace{5pt}\text{lim}_{q\in\mathcal{D}_{X}}G|_{q}(q^{-1}(W))=\widetilde{D}G(W)$,
\end{center}
and consider $\widetilde{D}\alpha_{G}(W)\circ\beta_{\widetilde{D}G}(W)(\zeta)$. 
By the definition, we see that $\beta_{\widetilde{D}G}(W)$ sends the given element $\zeta$ to 
the element 
$$
\xymatrix@R=10pt{
\Biggl(\biggl(\Bigl[\bigl[(s_{qW})_{q\in\mathcal{D}_{X}} \bigr]_{\eta_{X} q'((\eta_{X} q')^{-1}(W))} \Bigr]_{w}\biggr)_{w\in(\eta_{X} q')^{-1}(W)}\Bigg)_{q'\in\mathcal{D}_{X}}
}
$$
in $\widetilde{D}\eta^{*}_{X}\widetilde{C}\widetilde{D}G(W)$.
Note that 
\begin{align*} 
\widetilde{D}\eta^{*}_{X}\widetilde{C}\widetilde{D}G(W)&=\text{lim}_{q'\in\mathcal{D}_{X}}\hspace{2pt}\eta^{*}_{X}\widetilde{C}\widetilde{D}G|_{q'}(q'^{-1}(W)),\\
&= \text{lim}_{q'\in\mathcal{D}_{X}}\hspace{2pt}\widetilde{C}\widetilde{D}G|_{\eta_{X} q'}((\eta_{X} q')^{-1}(W))\\
&= \text{lim}_{q'\in\mathcal{D}_{X}}\biggl\{(a_{w})_{w}\in\prod_{w\in(\eta_{X} q')^{-1}(W)}((\eta_{X} q'|_{(\eta_{X}q')^{-1}(W)})^{*}\widetilde{D}G)_{w}|(**)\biggr\},
\end{align*}
where $a_{w}$ corresponds to $\Bigl[\bigl[(s_{qW})_{q\in\mathcal{D}_{X}} \bigr]_{\eta_{X} q'((\eta_{X} q')^{-1}(W))} \Bigr]_{w}$. We apply  $\widetilde{D}\alpha_{G}(W)$ to the image of $\beta_{\widetilde{D}G}(W)$. As before, we obtain its projection to each $q'$ and make a compatible family. In this case, the family $\{s_{q'W}\in G|_{q'}(q'^{-1}(W))\}$, that consists of an unique element, is one of the family we want. That is  because the germs composing the image of $\beta_{\widetilde{D}G(W)}$ can be represented by a section. From the compatible family, we obtain the section $s_{q'W}\in G|_{q'}(q'^{-1}(W))$. Then, we see that 
$\widetilde{D}\alpha_{G}(W)\circ\beta_{\widetilde{D}G}(W)((s_{qW}\in G|_{q}(q^{-1}(W)))_{q\in\mathcal{D}_{X}})=(s_{qW}\in G|_{q}(q^{-1}(W)))_{q\in\mathcal{D}_{X}}$. Hence, the right triangle is also commutative.
\end{proof}

\subsection{The second main theorem and the corollary }\label{sect:Thm_A_2}
We describe the other main theorem.

Let $(X,Open(X))$ be a topological space and $\varepsilon$ the counit associated with the adjunction $(C, D)$. The continuous map $$\varepsilon_{X}:=\varepsilon_{(X,Open(X))}:DC(X,Open(X))\to (X,Open(X))$$ 
 is  the identity map on the underlying set $X$.

\begin{defn}(The functor $\varepsilon_{X*}$)\label{defn:epsilon_*}
 Let $Open(DCX)$ and $Open(X)$ denote the categories of open subsets of $DC(X,Open(X))$ and $(X,Open(X))$, respectively. The map $\varepsilon_X$ induces the functor $\varepsilon_{X*}:Sh_{\mathsf{Top}}^{\mathcal{C}}Open(DCX)\to Sh_{\mathsf{Top}}^{\mathcal{C}}Open(X)$ defined by
\begin{center}
$\varepsilon_{X*}\theta_{X}:=\theta_{X}\circ\varepsilon_{X}^{-1}$, 
$\hspace{5pt}(\varepsilon_{X*}\mu)_{U}=\mu_{\varepsilon_{X}^{-1}(U)}:\theta_{X}(\varepsilon_{X}^{-1}(U))\to\theta'_{X}(\varepsilon_{X}^{-1}(U))$
\end{center}
for each $U\in (X,Open(X))$, $\theta_X\in Sh_{\mathsf{Top}}^{\mathcal{C}}Open(DCX)$ and a morphism $\mu:\theta_X\to\theta'_X$ in  $Sh_{\mathsf{Top}}^{\mathcal{C}}Open(DCX)$. The presheaf $\varepsilon_{X*}\theta_{X}$ is the pushforward sheaf of $\theta_X$ along $\varepsilon_X$ and then  $\varepsilon_{X*}\theta_{X}$ is a sheaf. 
\end{defn}

\begin{thm}\label{thm:main,epsilon.ver}
 \text{\em (The adjunction  $\widetilde{C}\dashv\varepsilon_{X*}\widetilde{D}$)}
Let $Open(X)$ be a category of open subsets of topological space $X$ and $\varepsilon_{X*}$ the functor defined in Definition \ref{defn:epsilon_*}. Then, there is an adjunction 
$$
\xymatrix@C=8pt@R=18pt{
Sh_{\mathsf{Top}}^{\mathcal{C}}Open(X)\ar[rr]^{\widetilde{C}}&&Sh_{\mathsf{Diff}}^{\mathcal{C}}\mathcal{D}_{CX}\ar[ld]^{\widetilde{D}}\\
&Sh_{\mathsf{Top}}^{\mathcal{C}}Open(DCX). \ar[lu]^{\varepsilon_{X*}}\ar@{}[u]|{\perp}&
}
$$
\end{thm}

In order to prove Theorem \ref{thm:main,epsilon.ver}, we show that the triangles 
$$
 \xymatrix@C=25pt@R=20pt
 {
\widetilde{C}\theta_X(p)\ar[r]^-{\widetilde{C}\delta_{\theta_X}(p)}\ar[rd]_{id_{\widetilde{C}\theta_X}(p)}&\widetilde{C}(\varepsilon_{X*}\widetilde{D})\widetilde{C}\theta_X(p)\ar[d]^{\gamma_{\widetilde{C}\theta_X}(p)} \ \ \ and&\hspace{-0.9cm}\varepsilon_{X*}\widetilde{D}G(W)\ar[r]^(.4){\delta_{\varepsilon_{X*}\widetilde{D}G}(U)}\ar[rd]_{id_{\varepsilon_{X*}\widetilde{D}G}(U)}&(\varepsilon_{X*}\widetilde{D})\widetilde{C}(\varepsilon_{X*}\widetilde{D})G(U)\ar[d]^{\varepsilon_{X*}\widetilde{D}\gamma_{G}(U)}\\
&\widetilde{C}\theta_X(p)&&\varepsilon_{X*}\widetilde{D}G(U)
}
$$
are commutative with natural transformations $\gamma$ and $\delta$.

\begin{const}\label{cons:alpha.epsilon}
(The natural transformation $\gamma$) 
We define a natural transformation $\gamma:\widetilde{C}(\varepsilon_{X*}\widetilde{D})\longrightarrow id_{Sh_{\mathsf{Diff}}^{\mathcal{C}}\mathcal{D}_{CX}}$. For $G\in Sh_{\mathsf{Diff}}^{\mathcal{C}}\mathcal{D}_{CX}$ and $q\in \mathcal{D}_{CX}$, 
we have 
\begin{eqnarray*}
\widetilde{C}(\varepsilon_{X*}\widetilde{D})G(q)&=&a_{U_q}q^*(\varepsilon_{X*}\widetilde{D}G)(U_q)\\
&=&\{(b_v)_{v\in U_q}\in\prod_{v\in U_q}(q^*\varepsilon_{X*}\widetilde{D}G)_v| \ (**) \ \},
\end{eqnarray*}
where $(**)$ is the sheafification condition in Definition \ref{defn:top.sheafification}.  Let $(q^*\varepsilon_{X*}\widetilde{D}G)_v$ denote the stalk at $v$ of $q^*\varepsilon_{X*}\widetilde{D}G$, namely, 
\begin{eqnarray}\label{eta.ver.element}
(q^*\varepsilon_{X*}\widetilde{D}G)_v=\text{colim}_{v\in U}\hspace{2pt}\text{colim}_{q(U)\subset V}\hspace{2pt}\text{lim}_{p\in\mathcal{D}_{CX}}G|_p(p^{-1}\varepsilon^{-1}_X(V)).
\end{eqnarray}
Thus, an element $(b_v)_v$ in $\widetilde{C}(\varepsilon_{X*}\widetilde{D})G(q)$ has the form 
\[
\biggl(\Bigl[\Bigl[\Bigl( s_{pV}\in G|_p(p^{-1}\varepsilon_X^{-1}(V)) \Bigr)_{p\in\mathcal{D}_{CX}}\Bigr]_{q(U)}   \Bigr]_v\biggr)_{v\in U_q}.
\]
 Every element of $\widetilde{C}(\varepsilon_{X*}\widetilde{D})G(q)$ satisfies the sheafification condition $(**)$ described in Definition \ref{defn:top.sheafification}. Then,  for all $w\in U_q$, there exist an open subset  $\mathcal{U}^w\subset U_q$ and a section $\bar{\sigma}^w\in q^*\varepsilon_{X*}\widetilde{D}G(\mathcal{U}^w)$ such that for $u\in \mathcal{U}^w$, $[\bar{\sigma}^w]_u=b_u$ in $(q^*\varepsilon_{X*}\widetilde{D}G)_u$. Let $\sigma'^w=(\sigma'^w_q\in G|_q(q^{-1}\varepsilon_X^{-1}(\mathcal{V}^w)))_{q\in\mathcal{D}_{CX}}$ be a representative of $\bar{\sigma}^w$. 
 
 By using the relation of $\mathcal{U}^w\subset q^{-1}\varepsilon_X^{-1}(\mathcal{V}^w)$, we have a family $\{\sigma_q^w\in G|_q(\mathcal{U}^w)\}_{w\in U_q}$, where $\sigma_q^w:=\sigma'^w_q|_{\mathcal{U}^w}$. Since the functor $G|_q$ is a sheaf on $Open(U_q)$, this compatible family gives a unique section $\Sigma_q$ of $G|_q(U_q)=G(q)$ such that $\Sigma_q|_{\mathcal{U}^w}=\sigma^w_q$ for each $w\in U_q$. We define $\gamma_{G(q)} : 
 \widetilde{C}(\varepsilon_{X*}\widetilde{D})G(q) \longrightarrow G(q)$ by
 \begin{eqnarray}
 \gamma_{G(q)}\Biggl(\biggl(\Bigl[\Bigl[\Bigl( s_{pV}\in G|_p(p^{-1}\varepsilon_X^{-1}(V)) \Bigr)_{p\in\mathcal{D}_{CX}}\Bigr]_{q(U)}   \Bigr]_v\biggr)_{v\in U_q}\Biggl)=\Sigma_q.
 \end{eqnarray}
\end{const}

\begin{lem}
The family $\{\sigma_q^w\in G|_q(\mathcal{U}^w)\}_{w\in U_q}$ in Construction \ref{cons:alpha.epsilon} is a compatible family.
\end{lem}
\begin{proof}
The argument as in the proof of Lemmas \ref{lem:compatible} enables us to conclude that the family is compatible. 
\end{proof}

\begin{lem}\label{lem:well-def.gamma}
The map of $\gamma_{G}(q):\widetilde{C}(\varepsilon_{X*}\widetilde{D})G(q)\to G(q)$ in  Construction \ref{cons:alpha.epsilon} is well defined.
\end{lem}
\begin{proof}
In order to prove the assertion, 
the same argument as in the proof of Lemma \ref{lem:well-def of alpha} is applicable. Let $([[(s_{pV}\in G|_p(p^{-1}\varepsilon_X^{-1}(V)))_{p\in\mathcal{D}_X}]_{q(U)}]_v)_{v\in U_q}$ be an element of $\widetilde{C}(\varepsilon_{X*}\widetilde{D})G(q)$ in Construction \ref{cons:alpha.epsilon}. We suppose that the given element induces
two family $\{\sigma'^w_q\in G|_q (q^{-1}\varepsilon_X^{-1}(\mathcal{V}^w))\}_w$ and $\{\mu'^w_q\in G|_q (q^{-1}\varepsilon_X^{-1}(\mathcal{V'}^w))\}_w$ and then they give two compatible families $\{\sigma^w_q\in G|_q(\mathcal{U}^w)\}_w$ and $\{\mu^w_q\in G|_q(\mathcal{U'}^w)\}_w$, respectively. It suffices to check   that $\mu_q^w|_{\mathcal{U}^w\cap\mathcal{U'}^w}=\sigma_q^w|_{\mathcal{U}^w\cap\mathcal{U'}^w}$ for $w\in U_q$.  
In the stalks which define $\widetilde{C}\varepsilon_{X*}\widetilde{D}G(q)$, 
we have
\begin{center}
$[[\mu'^w_q]_{q(\mathcal{U'}^w)}]_x=[[s_{qV}]_{q(U)}]_x=[[\sigma'^w_q]_{q(\mathcal{U}^w)}]_x$
\end{center}
for each element $x\in\mathcal{U}^w\cap\mathcal{U'}^w$. Then, we have an open subset $x\in O_x\subset\mathcal{U}^w\cap\mathcal{U'}^w\cap U$ with $[\mu'^w_q]_{q(O_x)}=[\sigma'^w_q]_{q(O_x)}$. The equality yields an open subset $U_x$ of $Open(X)$ satisfying $q(O_x)\subset U_x$ and $\mu'^w_q|_{q^{-1}\varepsilon_X^{-1}(U_x\cap\mathcal{V}^w\cap\mathcal{V'}^w)}=\sigma'^w_q|_{q^{-1}\varepsilon_X^{-1}(U_x\cap\mathcal{V}^w\cap\mathcal{V'}^w)}$. Let $\mathcal{W}_x$ be the intersection $\mathcal{U}^w\cap\mathcal{U'}^w\cap q^{-1}\varepsilon_X^{-1}(U_x\cap\mathcal{V}^w\cap\mathcal{V'}^w)$.  We have an  equality $\mu_q^w|_{\mathcal{W}_x}=\sigma_q^w|_{\mathcal{W}_x}$. The family $\{\mathcal{W}_x\}_{x\in \mathcal{U}^w\cap\mathcal{U'}^w}$ covers $\mathcal{U}^w\cap\mathcal{U'}^w$ and hence $\mu_q^w|_{\mathcal{U}^w\cap\mathcal{U'}^w}=\sigma_q^w|_{\mathcal{U}^w\cap\mathcal{U'}^w}$. 
\end{proof}

\begin{lem}
\text{\rm (The naturality of $\gamma$)}
The construction of $\gamma_{G(q)}$ gives rise to a natural transformation $\gamma:\widetilde{C}(\varepsilon_{X*}\widetilde{D})\longrightarrow id_{Sh_{\mathsf{Diff}}^{\mathcal{C}}\mathcal{D}_{CX}}$.
\end{lem}
\begin{proof}
Let $\lambda:G\to G'$ be a natural transformation. It suffices to show that the following diagram 
\begin{eqnarray}\label{diag.nat.gamma}
\xymatrix@C=40pt{
\widetilde{C}(\varepsilon_{X*}\widetilde{D})G(q)\ar[r]^(.6){\gamma_{G(q)}}\ar[d]_{(\widetilde{C}(\varepsilon_{X*}\widetilde{D})\lambda)_q}&G(q)\ar[d]^{\lambda_q}\\
\widetilde{C}(\varepsilon_{X*}\widetilde{D})G'(q)\ar[r]_(.6){\gamma_{G'(q)}}&G'(q)
}
\end{eqnarray}
is commutative.
First, we consider the map $(\widetilde{C}(\varepsilon_{X*}\widetilde{D})\lambda)_q$ whose construction is explained with the diagram 
\begin{eqnarray}\label{diag:cons.CepDlam}
\xymatrix@C=2pt@R=8pt{
 && & & &(q^{*}\varepsilon_{X*}\widetilde{D}G')_{v''}\\
 &&q^{*}\varepsilon_{X*}\widetilde{D}G'(\mathcal{U}^w)\ar[rrru]\ar[rd]\ar[rr] & &(q^{*}\varepsilon_{X*}\widetilde{D}G')_{v'} &\\
 && &(q^{*}\varepsilon_{X*}\widetilde{D}G')_{v} & &(q^{*}\varepsilon_{X*}\widetilde{D}G)_{v''}\ar@{.>}[uu]_{_{(q^{*}\varepsilon_{X*}\widetilde{D}\lambda)_{v''}}}\\
 \widetilde{D}G'(\varepsilon_X^{-1}(\mathcal{V}^w))\ar[rruu]&&q^{*}\varepsilon_{X*}\widetilde{D}G(\mathcal{U}^w)\ar[uu]_(.3){_{(q^{*}\varepsilon_{X*}\widetilde{D}\lambda)_{\mathcal{U}^w}}}\ar[rd]\ar[rr]\ar[rrru]& &(q^{*}\varepsilon_{X*}\widetilde{D}G)_{v'}\ar@{.>}[uu]_(.7){_{(q^{*}\varepsilon_{X*}\widetilde{D}\lambda)_{v'}}}&\\
 && &(q^{*}\varepsilon_{X*}\widetilde{D}G)_{v}\ar@{.>}[uu]_(.25){_{(q^{*}\varepsilon_{X*}\widetilde{D}\lambda)_{v}}}&&\\
\widetilde{D}G(\varepsilon_{X}^{-1}(\mathcal{V}^w))\ar[uu]^{\widetilde{D}\lambda_{\varepsilon_{X}^{-1}(\mathcal{V}^w)}}.\ar[rruu]&&&&&
}
\end{eqnarray}
The maps $\widetilde{D}\lambda_{\varepsilon^{-1}_{X}(\mathcal{V}^w)}$, $(q^*\varepsilon_{X*}\widetilde{D}\lambda)_{\mathcal{U}^w}$ and $(q^*\varepsilon_{X*}\widetilde{D}\lambda)_v$ are induced by the universalities of $\widetilde{D}G'(\varepsilon^{-1}_X(\mathcal{V}^w))$, $q^*\varepsilon_{X*}\widetilde{D}G(\mathcal{U}^w)$ and each $(q^*\varepsilon_{X*}\widetilde{D}G)_v$, respectively. 
Then, the map $(\widetilde{C}(\varepsilon_{X*}\widetilde{D})\lambda)_q$ 
is defined by 
\begin{center}
$(\widetilde{C}(\varepsilon_{X*}\widetilde{D})\lambda)_q=\prod_{v\in U_q}(q^{*}\varepsilon_{X*}\widetilde{D}\lambda)_{v}$. 
\end{center}
Let $([[( s_{pV}\in G|_p(p^{-1}\varepsilon_X^{-1}(V)) )_{p\in\mathcal{D}_{CX}} ]_{q(U)}]_v)_{v\in U_q}$ be an element of $\widetilde{C}(\varepsilon_{X*}\widetilde{D})G(q)$. 
By considering the germs which give 
$\widetilde{C}(\varepsilon_{X*}\widetilde{D})G(q)$, we see that for each $w\in U_q$, there exist an open set $\mathcal{U}^w$ containing $w$ and a section $\bar{\sigma}^w\in q^*\varepsilon_{X*}\widetilde{D}G(\mathcal{U}^w)$ such that $[\bar{\sigma}^w]_v=[[( s_{pV}\in G|_p(p^{-1}\varepsilon_{X}^{-1}(V)) )_{p\in\mathcal{D}_{CX}} ]_{q(U)}]_v$ for all $v\in\mathcal{U}^w$. We have a representative $(\sigma'^w_p\in G|_p(p^{-1}\varepsilon_{X}^{-1}(\mathcal{V}^w)))_{p\in\mathcal{D}_{CX}}$ in the class of $\bar{\sigma}^w$, where $\mathcal{V}^w$ is an open set of $X$ with $q(\mathcal{U}^w)\subset \mathcal{V}^w$. 
 A compatible family $\{\sigma^w_q:=\sigma'^w_q|_{\mathcal{U}^w}\}_{w\in U_q}$ induced by the element $(([[ s_{qV}\in G|_q(q^{-1}\varepsilon_X^{-1}(V))  ]_{q(U) }]_v)_{v\in U_q})$ yields a section $\Sigma_q\in G(q)$. Its restriction $\Sigma_q|_{\mathcal{U}^w}$ is equal to $\sigma_q^w$ for every $w\in U_q$. We conclude that 
\begin{center}
$\lambda_q\circ\gamma_{G}(q)(([[( s_{pV}\in G|_p(p^{-1}\varepsilon_X^{-1}(V)) )_{p\in\mathcal{D}_{CX}} ]_{q(U)}]_v)_{v\in U_q})=\lambda_q(\Sigma_q)$
\end{center}
and the image satisfies the equalities    $\lambda_q(\Sigma_q)|_{\mathcal{U}^w}=\lambda_{q|_{\mathcal{U}^w}}(\Sigma_q|_{\mathcal{U}^w})=\lambda_{\mathcal{U}^w}(\sigma^w_q)$ for $w\in U_q$. Observe that $\lambda_{q|_{\mathcal{U}^w}}$ in $Sh_{\mathsf{Diff}}^{\mathcal{C}}\mathcal{D}_{CX}$ is identified with $\lambda_{\mathcal{U}^w}$ in $Sh_{\mathsf{Top}}^{\mathcal{C}}Open(U_q)$. 

On the other hand, 
with the left vertical arrow 
in the diagram \ref{diag.nat.gamma}, we see that 
\begin{align*}
  &  (\widetilde{C}(\varepsilon_{X*}\widetilde{D})\lambda)_q(([[( s_{pV} )_{p\in\mathcal{D}_{CX}}]_{q(U) }]_v)_{v\in U_q} )\\
=& \prod_{v\in U_q}(q^{*}\varepsilon_{X*}\widetilde{D}\lambda)_{v}(([[( \sigma'^v_p )_{p\in\mathcal{D}_{CX}}]_{q(\mathcal{U}^v) }]_v)_{v\in U_q})\\
=& ((q^{*}\varepsilon_{X*}\widetilde{D}\lambda)_{v}([[( \sigma'^v_p )_{p\in\mathcal{D}_{CX}}]_{q(\mathcal{U}^v) }]_v))_{v\in U_q}\\
=&([(q^*\varepsilon_{X*}\widetilde{D}\lambda)_{\mathcal{U}^v}([(\sigma'^v_p)_{p\in\mathcal{D}_{CX}}]_{q(\mathcal{U}^v)})]_v)_{v\in{U_q}}\\
=&([[\widetilde{D}\lambda_{\varepsilon^{-1}_X(\mathcal{V}^v)}((\sigma'^v_p)_{p\in\mathcal{D}_{CX}})]_{q(\mathcal{U}^v) }]_v)_{v\in U_q}\\
=&([[(\lambda_{p^{-1}\varepsilon_X^{-1}(\mathcal{V}^v)} (\sigma'^v_p) )_{p\in\mathcal{D}_{CX}}]_{q(\mathcal{U}^v) }]_v)_{v\in U_q}.
\end{align*}
Furthermore, the map $\gamma_{G'(q)}$ gives a projection to the component with the plot $q$, 
$([[\lambda_{q^{-1}\varepsilon_X^{-1}(\mathcal{V}^v)} (\sigma'^v_q) ]_{q(\mathcal{U}^v)}]_v)_{v\in U_q}=([[\lambda_{\mathcal{U}^v}(\sigma^v_q)]_{q(\mathcal{U}^v)}]_v)_{v\in U_q}$, 
a compatible family $\{\lambda_{\mathcal{U}^v}(\sigma^v_q)\in G'|_q(\mathcal{U}^v)\}_{v\in U_q}$ and a section $\Sigma'_q\in G'(q)$ with $\Sigma'_q|_{\mathcal{U}^v}=\lambda_{\mathcal{U}^v}(\sigma^v_q)$ for every $v\in U_q$. Due to the sheafification property(the uniqueness of extensions), we have $\lambda_q(\Sigma_q)=\Sigma'_q$.
\end{proof}

\begin{const}\label{cons:delta}
(The natural transformation $\delta$) 
We construct a natural transformation $\delta:id_{Sh_{\mathsf{Top}}^{\mathcal{C}}Open(X)}\to(\varepsilon_{X*}\widetilde{D})\widetilde{C}$. Let $\theta_X$ be an element of $Sh_{\mathsf{Top}}^{\mathcal{C}}Open(X)$ and $U$ an open subset of $Open(X)$. First, we see that
\begin{eqnarray*}
(\varepsilon_{X*}\widetilde{D})\widetilde{C}\theta_X(U)&=&\widetilde{D}\widetilde{C}\theta_{X}(\varepsilon^{-1}_X(U)) 
=\text{lim}_{q\in\mathcal{D}_{CX}}\widetilde{C}\theta_X|_q(q^{-1}\varepsilon^{-1}_X(U))\\
&=&\text{lim}_{q\in\mathcal{D}_{CX}}\{\hspace{2pt}(b_v)_v\in \prod_{v\in q^{-1}\varepsilon^{-1}_X(U)}(q^*\theta_X)_v|(**)\hspace{2pt}\},
\end{eqnarray*}
where $(**)$ is the sheafification condition in Definition \ref{defn:top.sheafification} and $(q^*\theta_X)_v$ is the stalk of $q^*\theta_X$ at $v$. Observe that $(q^*\theta_X)_v=\text{colim}_{v\in W}\hspace{2pt}\text{colim}_{q(W)\subset V,\hspace{2pt}V\in Open(X)}\hspace{2pt}\theta_X(V)$. 

Let $s_U$ be an element of $\theta_X(U)$ and $q$ a plot of $\mathcal{D}_{CX}$. To construct a map $\delta_{\theta_X}(U):\theta_X(U)\to((\varepsilon_{X*}\widetilde{D})\widetilde{C}\theta_X)(U)$, we make an element of $((\varepsilon_{X*}\widetilde{D})\widetilde{C}\theta_X)(U)$ with $s_U$. Because of the relation $q(q^{-1}\varepsilon^{-1}_X(U))=\varepsilon^{-1}_X(U)\subset U$, we have a family of germs $([[s_U]_{q(q^{-1}\varepsilon^{-1}_X(U))}]_v)_{v\in q^{-1}\varepsilon^{-1}_X(U)}\in \prod_{v\in q^{-1}\varepsilon^{-1}_X(U)}(q^*\theta_X)_v$. Moreover, it follows from the equalities 
\[p(p^{-1}\varepsilon_X^{-1}(U))=qf(f^{-1}q^{-1}\varepsilon^{-1}_X(U))=q(q^{-1}\varepsilon^{-1}_X(U))
\]
that the map $\widetilde{C}\theta_X(f|_{f^{-1}q^{-1}\varepsilon^{-1}_X(U)})$ sends the element $([[s_U]_{q(q^{-1}\varepsilon^{-1}_X(U))}]_v)_{v\in q^{-1}\varepsilon^{-1}_X(U)}$ to another element $([[s_U]_{p(p^{-1}\varepsilon^{-1}_X(U))}]_u)_{u\in p^{-1}\varepsilon^{-1}_X(U)}\in \prod_{u\in p^{-1}\varepsilon_X^{-1}(U)}(p^*\theta_X)_u$ for a morphism $f:p\to q$ of $\mathcal{D}_{CX}$, 
where $f(u)=v$.  The commutative diagram 
$$
\xymatrix@C=18pt@R=5pt{
 && & & &(p^{*}\theta_X)_{u}\\
 &&\text{colim}_{p(p^{-1}\varepsilon^{-1}_X(U))}\theta_X(U)\ar[rrru]\ar[rd]\ar[rr] & &(p^{*}\theta_X)_{u'} &\\
 && &(p^{*}\theta_X)_{u''} & &( q^{*}\theta_X)_{v}\ar@{.>}[uu]\\
 &&\text{colim}_{q(q^{-1}\varepsilon^{-1}_X(U))}\theta_X(U)\ar[uu]_(.3){(f|_{p^{-1}\varepsilon^{-1}_X(U)})^{*}}\ar[rd]\ar[rr]\ar[rrru]& &( q^{*}\theta_X)_{v'}\ar@{.>}[uu]&\\
 && &(q^{*}\theta_X)_{v''}\ar@{.>}[uu]&&\\
 \theta_X(U)\ar[rruu]\ar[rruuuu]&&&&&
}
$$
shows this situation, where $f(u)=v, f(u')=v'$ and $ f(u'')=v''$. Therefore, we have an element  $(([[s_U]_{q(q^{-1}\varepsilon^{-1}_X(U))}]_v)_{v\in q^{-1}\varepsilon^{-1}_X(U)})_{q\in\mathcal{D}_{CX}}\in(\varepsilon_{X*}\widetilde{D})\widetilde{C}\theta_X(U)$. Then, we define $\delta_{\theta_X(U)}$ by
\begin{eqnarray}
\delta_{\theta_X(U)}(s_U)= (([[s_U]_{q(q^{-1}\varepsilon^{-1}_X(U))}]_v)_{v\in q^{-1}\varepsilon^{-1}_X(U)})_{q\in\mathcal{D}_{CX}}.
\end{eqnarray}
\end{const}

\begin{lem}\label{lem.nat.delta}
\text{\rm (The naturality of $\delta$)}
The construction of \ref{cons:delta} gives rise to a natural transformation $\delta:id_{Sh_{\mathsf{Top}}^{\mathcal{C}}Open(X)}\to (\varepsilon_{X*}\widetilde{D})\widetilde{C}$.
\end{lem}

\begin{proof}
It suffices to show that the diagram
\begin{eqnarray}\label{diag.nat.delta}
\xymatrix@C=40pt@R=20pt{
\theta_X(U)\ar[r]^(.4){\delta_{\theta_X(U)}}\ar[d]_{\lambda_U}&(\varepsilon_{X*}\widetilde{D})\widetilde{C}\theta_X(U)\ar[d]^{(\varepsilon_{X*}\widetilde{D}\widetilde{C}\lambda)_U}\\
\theta'_X(U)\ar[r]_(.4){\delta_{\theta'_X(U)}}&(\varepsilon_{X*}\widetilde{D})\widetilde{C}\theta_X(U)
}
\end{eqnarray}
 is commutative for each morphism $\lambda:\theta_X\to\theta'_X$ and $U\in Open(X)$. 
We recall the definition of $\delta$. In the following diagram (\ref{diag.nat.delta2}), arrows with stars and rhombuses
give rise to the morphisms $\delta_{\theta_X}(U)$ and $\delta_{\theta'_X}(U)$, respectively.  
\begin{eqnarray}\label{diag.nat.delta2}
\xymatrix@C=8pt@R=8pt{
 && & & &(q^{*}\theta_X)_{v''}\ar@{.>}[dd]\\
&&\text{colim}_{q(q^{-1}\varepsilon^{-1}_X(U))\subset V}\theta_X(V)\ar[dd]\ar[rrru]^{\star}\ar[rd]^{\star}\ar[rr]^{\star} & &(q^{*}\theta_X)_{v'}\ar@{.>}[dd] &\\
 && &(q^{*}\theta_X)_{v}\ar@{.>}[dd] & &(q^{*}\theta'_X)_{v''}\\
 \theta_X(U)\ar[dd]_{\lambda_U}\ar[rruu]^{\star}&&\text{colim}_{q(q^{-1}\varepsilon^{-1}_X(U)\subset V}\theta'_X(V)\ar[rd]_{\diamond}\ar[rr]_(.55){\diamond}\ar[rrru]_(.6){\diamond}& &(q^{*}\theta'_X)_{v'}&\\
 && &(q^{*}\theta'_X)_{v}&&\\
\theta'_X(U)\ar[rruu]_{\diamond}&&&&&
}
\end{eqnarray}
Moreover, the dotted arrows induce 
the map $(\varepsilon_{X*}\widetilde{D}\widetilde{C}\lambda)_U$.  The commutativity of the diagram (\ref{diag.nat.delta2}) yields that of the diagram (\ref{diag.nat.delta}).
\end{proof}

\begin{proof}[Proof of Theorem \ref{thm:main,epsilon.ver}.]
First, we show the commutativity of the diagram
\begin{eqnarray}
 \xymatrix@C=30pt@R=20pt
 {
\widetilde{C}\theta_X(p)\ar[r]^(.4){\widetilde{C}\delta_{\theta_X}(p)}\ar[rd]_{id_{\widetilde{C}\theta_X}(p)}&\widetilde{C}(\varepsilon_{X*}\widetilde{D})\widetilde{C}\theta_X(p)\ar[d]^{\gamma_{\widetilde{C}\theta_X}(p)}\\
&\widetilde{C}\theta_X(p).
}
\end{eqnarray}
Let $([[s_V]_{p(U)}]_u)_{u\in U_p}$ be an element of $\widetilde{C}\theta_X(p)$, where $s_V\in\theta_X(V)$, $u\in U$ and $p(U)\subset V$. 
By the definition of $\delta$, we have  
\begin{center}
$\widetilde{C}\delta_{\theta_{X}}(p)\bigl(([[s_V]_{p(U)}]_u)_{u\in U_p}\bigr)=\biggl( \biggl[ \Bigl[ \bigl( \bigl( \bigl[ [s_V]_{q(q^{-1}\varepsilon_{X}^{-1}(V))} \bigr]_w \bigr)_{w\in q^{-1}\varepsilon_X^{-1}(V)} \bigr)_{q\in\mathcal{D}_{CX}} \Bigl]_{p(U)}\biggr]_u \biggr)_{u\in U_p}.$
\end{center}
Moreover, we apply $\gamma_{\widetilde{C}\theta_X}(p)$ to this element. We take a compatible family 
\begin{center}
$\Bigl\{\bigl(\bigl[ [s_V]_{p(p^{-1}\varepsilon_{X}^{-1}(V))} \bigr]_v \bigr)_{v\in \mathcal{U}^w}\Bigr\}_{w\in U_p}$. 
\end{center}
Using the sheaf property of $\widetilde{C}\theta_X(p)$, we obtain  $\gamma_{\widetilde{C}\theta_X}(p)\circ\widetilde{C}\delta_{\theta_{X}}(p)=\Sigma_p\in\widetilde{C}\theta_X(p)$ such that $\Sigma_p|_{\mathcal{U}^w}=\bigl(\bigl[[s_V]_{p(p^{-1}\varepsilon_{X}^{-1}(V))}\bigr]_v\bigr)_{v\in\mathcal{U}^w}$.  On the other hand, for every $w$, the restriction of $([[s_V]_{p(U)}]_u)_{u\in U_p}$ to $\mathcal{U}^w$ is equal to $([[s_V]_{p(U)}]_u)_{u\in \mathcal{U}^w}$.
By the sheaf condition again, we have $\Sigma_p=([[s_V]_{p(U)}]_u)_{u\in U_p}$.

We consider the other triangle identity which yields the commutativity of the diagram  
\begin{eqnarray}
 \xymatrix@C=30pt@R=20pt
 {
 \varepsilon_{X*}\widetilde{D}G(U)\ar[r]^(.4){\delta_{\varepsilon_{X*}\widetilde{D}G}(U)}\ar[rd]_{id_{\varepsilon_{X*}\widetilde{D}G}(U)}&(\varepsilon_{X*}\widetilde{D})\widetilde{C}(\varepsilon_{X*}\widetilde{D})G(U)\ar[d]^{\varepsilon_{X*}\widetilde{D}\gamma_{G}(U)}\\
 &\varepsilon_{X*}\widetilde{D}G(U).
 }
 \end{eqnarray}
 Let $(s_{qU}\in G|_q(q^{-1}\varepsilon^{-1}_X(U)))_{q\in\mathcal{D}_{CX}}$ be an element of $\varepsilon_{X*}\widetilde{D}G(U)$. By the definition of $\delta$, we have
 \begin{center}
 $\delta_{\varepsilon_{X*}\widetilde{D}G}(U)((s_{qU})_{q\in\mathcal{D}_{CX}})=\biggl( \biggl( \Bigl[\big[ ( s_{qU} )_{q\in\mathcal{D}_{CX}} \bigr]_{p(p^{-1}\varepsilon_X^{-1}(U))} \Bigr]_u \biggr)_{u\in p^{-1}\varepsilon_X^{-1}(U)} \biggr)_{p\in\mathcal{D}_{CX}}$.
 \end{center}
 By applying the map $\varepsilon_{X*}\widetilde{D}\gamma_{G}(U)$ to this image, we choose a compatible family $\{s_{pU}\in G|_p(p^{-1}\varepsilon^{-1}_X(U))\}$ for each $p\in\mathcal{D}_{CX}$. The construction of $\gamma$ enables us to conclude that   $\varepsilon_{X*}\widetilde{D}\gamma_{G}(U)\circ \delta_{\varepsilon_{X*}\widetilde{D}G}(U)((s_{qU})_{q\in\mathcal{D}_{CX}})=(s_{qU})_{q\in\mathcal{D}_{CX}}$.
\end{proof}

We recall the fact that $CDC=C$ and $DCD=D$; see Section \ref{sect:diff}. The particular cases, the functors $\eta_{C(X,Open(X))}$ and   
$\varepsilon_{D(X,\mathcal{D}_X)}$ are indeed the identity ones. Observe that both of $\eta_{X}$ and $\varepsilon_{X}$ are identity maps on the underlying set. 


\begin{lem}\label{lem:id.functor.eta}
For a topological space $(X,Open(X))$, let $\mathcal{D}_{CX}$ and $\mathcal{D}_{CDCX}$ be the categories of plots of $C(X,Open(X))$ and $CDC(X,Open(X))$, respectively. Then, the smooth map 
$$
\eta_{CX}:=\eta_{C(X,Open(X))}:C(X,Open(X))\to CDC(X,Open(X))
$$ 
induces the functor $\eta_{CX}^*:Sh_{\mathsf{Diff}}^{\mathcal{C}}\mathcal{D}_{CDCX}\to Sh_{\mathsf{Diff}}^{\mathcal{C}}\mathcal{D}_{CX}$. Moreover, the functor $\eta^*_{CX}$ is the identity functor.
\end{lem}
\begin{proof}
Let $F\in Sh_{\mathsf{Diff}}^{\mathcal{C}}\mathcal{D}_{CX}$. For a plot $p\in C(X,Open(X))$, the definition of $\eta$ gives rise to  equalities $\eta_{CX}^*F(p)=F(\eta_{CX}\circ p)=F(p)$. 
The result follows from the fact that $CDC=C$.  
\end{proof}

\begin{lem}\label{lem:id.functor.epsilon}
For a diffeological space $(X,\mathcal{D}_X)$, let $Open(DX)$ and $Open(DCDX)$ be the categories of open subsets of $D(X,\mathcal{D}_X)$ and $DCD(X,\mathcal{D}_X)$, respectively. Then, the continuous map 
$$
\varepsilon_{DX}:=\varepsilon_{D(X,\mathcal{D}_X)}:DCD(X,\mathcal{D}_X)\to D(X,\mathcal{D}_X)
$$ 
induces the functor $\varepsilon_{DX*}:=\varepsilon_{D(X,\mathcal{D}_X)*}:Sh_{\mathsf{Top}}^{\mathcal{C}}Open(DCDX)\to Sh_{\mathsf{Top}}^{\mathcal{C}}Open(DX)$. Moreover, the functor $\varepsilon_{DX*}$ is the identity functor.
\end{lem}
\begin{proof}
Let $G\in Sh_{\mathsf{Top}}^{\mathcal{C}}Open(DCDX)$ and $U\in D(X,\mathcal{D}_X)$. By the definition of $\varepsilon$, we have  equalities $\varepsilon_{DX*}G(U)=G(\varepsilon^{-1}_{DX}(U))=G(U)$. Moreover, thanks to the relation $DCD=D$, we have the result.
\end{proof}

\begin{cor}\label{cor:adjunction}
For a topological space $(Y,Open(Y))$ and  a diffeological space $(X,\mathcal{D}_X)$, let $\mathcal{D}_{CY}$ and $\mathcal{D}_{CDX}$ be the categories of $plots$ of $C(Y,Open(Y))$ and $CD(X,\mathcal{D}_X)$. Furthermore, let $Open(DX)$ and $Open(DCY)$ be the categories of open subsets of $D(X,\mathcal{D}_X)$ and $DC(Y,Open(Y))$, respectively. Then, 
one has adjunctions 
$$
\xymatrix@C=30pt@R=15pt{
Sh_{\mathsf{Top}}^{\mathcal{C}}Open(DX)\ar@<1.0ex>[r]^{\widetilde{C}}_{\perp}&Sh_{\mathsf{Diff}}^{\mathcal{C}}\mathcal{D}_{CDX}\ar@<1.7ex>[l]^{\widetilde{D}}  \ \ \   \text{and} \\
Sh_{\mathsf{Diff}}^{\mathcal{C}}\mathcal{D}_{CY}\ar@<.8ex>[r]^(.4){\widetilde{D}}_(.4){\top}&Sh_{\mathsf{Top}}^{\mathcal{C}}Open(DCY). \ar@<1.7ex>[l]^(.6){\widetilde{C}}  \ \ \ \ \ 
}
$$
\end{cor}
\begin{proof} The first adjunction follows from Theorem \ref{thm:main,epsilon.ver} and Lemma \ref{lem:id.functor.epsilon}. 
The second one is obtained by Theorem \ref{thm:main,eta.ver.} and 
Lemma \ref{lem:id.functor.eta}. 
\end{proof}

{\it Acknowledgments.} The author thanks Alireza Ahmadi and Jordan Watts for their useful comments on the first version of this work.


\appendix
\section{A diffeological sheaf as a topos}\label{sect:AppA}
We introduce a site given by a diffeology in \cite{AA}.  As a consequence, we see that 
the sheaf $Sh_{\mathsf{Diff}}^{\mathsf{Sets}}\mathcal{D}_{X}$ is a Grothendieck topos on the site. Thus, the diffeological sheaves are entirely incorporated in topos theory.  

\subsection{A Grothendieck topos}
We begin with the definition of a Grothendieck pretopology.

\begin{defn}\cite{SGA}
Let $\mathcal{E}$ be a category. A Grothendieck pretopology (a basis for a Grothendieck topology) on $\mathcal{E}$ is a condition Cov for each object $p$ of $\mathcal{E}$, which assigns a colloection $\text{Cov}(p)$ of families of morphisms with  fixed target $p$ and satisfies the following axioms.
\begin{enumerate}
\item[i)] For all objects $p$ in $\mathcal{E}$, the morphisms consisting of the families in  $\text{Cov}(p)$ have the fibre products along every morphism of $\mathcal{E}$.
\item[ii)] For all object $p$ of $\mathcal{E}$, all families $(p_{\alpha}\to p)_{\alpha\in A}$ in $\text{Cov}(p)$ and all morphism $q\to p$ ($q\in Ob(\mathcal{E})$), the family of pullbacks $(p_{\alpha}\times_{p}q\to q)_{\alpha\in A}$ is in $Cov(q)$.
\item[iii)] If $(p_{\alpha}\to p)_{\alpha\in A}$ is in $\text{Cov}(p)$ and $(p_{\beta_{\alpha}}\to p_{\alpha})_{\beta_{\alpha}\in B_{\alpha}}$ is in $\text{Cov}(p_{\alpha})$ for each $\alpha\in A$, then the family of the compositions $(p_{\beta_{\alpha}}\to p_{\alpha}\to p)_{\beta_{\alpha}\in B_{\alpha}, \hspace{2pt}\alpha\in A}$ is in $\text{Cov}(p)$.
\item[iv)] The family $(p\stackrel{id_p}{\to} p)$ is in $\text{Cov}(p)$.

\end{enumerate}
\end{defn}

 We may call the collection $\text{Cov}(p)$ and an element of $\text{Cov}(p)$ a {\it covering} and a {\it covering family}, respectively. The pair $(\mathcal{E}, \text{Cov})$ is called a {\it site}. $\mathsf{Sets}$-valued presheaf $F$ is a sheaf if $F$ satisfies the following condition (*) on a site $(\mathcal{E}, \text{Cov})$.

\begin{itemize}
\item[](*) : For an object $p$ of the site $\mathcal{E}$ and a covering family $\{i_{i}:p_{i}\to p\}_{i\in I_{p}}$ in $\text{Cov}(p)$, 
let $p_{i}\times_{p}p_{j}$ be the pull back of $p_{i}\to p$ and $p_{j}\to p$, $\rho_{ij} : p_{i}\times_{p}p_{j} \to p_i$ and 
$\rho_{ij}' : p_{i}\times_{p}p_{j} \to p_j$  the natural maps, respectively. 
Then, one has an equalizer diagram 
$$
 \xymatrix@C=40pt{
 F(p)\ar[r]^-{\prod F(i_{i})}&\prod_{i\in I_{p}}F(p_{i})\ar@<0.5ex>[r]^(.35){\pi_1}\ar@<-0.5ex>[r]_(.35){\pi_2}&\prod_{(i,j)\in I_{p}\times I_{p}}F(p_{i}\times_{p} p_{j}), 
 }
$$
where $\pi_1$ and $\pi_2$ are defined by $\prod_{i,j}F(\rho_{ij})$ and 
$\prod_{i,j}F(\rho'_{ij})$.
\end{itemize}

 We introduce 
a basis $\text{Cov}$ on the category $\mathcal{D}_{X}$.

\begin{prop}\cite{AA}
Let $(X,\mathcal{D}_{X})$ be a diffeological space and $\mathcal{D}_X$ the category of plots. For each plot $p\in\mathcal{D}_X$, let $Open(U_p)$ be the category of open subsets on $U_p$ and $\{U_i\}_{i\in I_p}$ an open cover of $U_p$. The family of inclusions $\{U_i\hookrightarrow U_p\}_{i\in I_p}$  gives the family of inclusions of $\mathcal{D}_X$ naturally, $\{p_i:=p|_{U_i}\hookrightarrow p\}_{i\in I_p}$. Let $\text{Cov}(p)$ be the set of families of these inclusions. Then, it gives rise to a Grothendieck topology.
\end{prop}
\begin{proof}
See \cite[Proposition 3.1.]{AA}.
\end{proof}



\subsection{A topos on a diffeological space}
The category of $\mathsf{Sets}$-valued sheaves on the site $\mathcal{J}_{X}:=(\mathcal{D}_X, \text{Cov})$ is denoted by $\mathscr{S}_{\mathsf{Diff}}^{\mathsf{Sets}}(X,\mathcal{J}_{X})$, namely, the full subcategory  of the functor category $\mathsf{Sets}^{(\mathcal{D}_X)^{\rm op}}$ on $\mathcal{D}_{X}$ with values in $\mathsf{Sets}$ consisting of sheaves. We observe that $Sh_{\mathsf{Diff}}^{\mathsf{Sets}}\mathcal{D}_{X}$ is also a full subcategory of 
$\mathsf{Sets}^{(\mathcal{D}_X)^{\text{op}}}=PSh_{\mathsf{Diff}}^{\mathsf{Sets}}\mathcal{D}_X$. 

The following result is already known by Alireza Ahmadi and Jordan Watts.

\begin{prop}\label{prop:topos_Sh}
Let $(X,\mathcal{D}_{X})$ be a diffeological space. 
The Grothendieck topos $\mathscr{S}_{\mathsf{Diff}}^{\mathsf{Sets}}(X,\mathcal{J}_{X})$ coincides with the category  $Sh_{\mathsf{Diff}}^{\mathsf{Sets}}\mathcal{D}_{X}$. 
\end{prop}
\begin{proof} For each plot $p \in \mathcal{D}_X$, the covering $\text{Cov}(p)$ contains a covering family of the form 
$\{ p_i \to p\}_{i\in I}$ with $\cup_{i\in I} U_{p_i} = U_p$. 
Then, it follows that every object of $\mathscr{S}_{\mathsf{Diff}}^{\mathsf{Sets}}(X,\mathcal{J}_{X})$ satisfies the condition of $Sh_{\mathsf{Diff}}^{\mathsf{Sets}}\mathcal{D}_{X}$. Thus, we have an inclusion functor $\mathscr{S}_{\mathsf{Diff}}^{\mathsf{Sets}}(X,\mathcal{J}_{X})\to Sh_{\mathsf{Diff}}^{\mathsf{Sets}}\mathcal{D}_{X}$.  


On the other hand, Let $F$ be an object of $Sh_{\mathsf{Diff}}^{\mathsf{Sets}}\mathcal{D}_{X}$. We show that the diagram 
$$
 \xymatrix@C=40pt{
 F(p)\ar[r]^-{\prod F(\iota_{\lambda})}&\prod_{\lambda\in \Lambda_{p}}F(p_{\lambda})\ar@<0.5ex>[r]^(.35){\pi_{1}}\ar@<-0.5ex>[r]_(.35){\pi_{2}}&\prod_{(\lambda,\lambda')\in \Lambda_{p}\times \Lambda_{p}}F(p_{\lambda}\times_{p} p_{\lambda'})  
}
$$
obtained  by a covering family $\{\iota_{\lambda}:p_{\lambda}\to p\}_{\lambda\in\Lambda_{p}}\in \text{Cov}(p)$ is the equalizer for a plot $p$. Note that all covering families consist of inclusions. Then, all elements of $Sh_{\mathsf{Diff}}^{\mathsf{Sets}}\mathcal{D}_X$ satisfy the condition of $\mathscr{S}_{\mathsf{Diff}}^{\mathsf{Sets}}(X, \mathcal{J})$.

\end{proof}

\begin{rem}
The result allows us to conclude that a sheaf in the sense in \cite{AA} is nothing but one of \cite{KWW}. 
\end{rem}


\section{The sheafification functor of presheaves on a diffeological space.}\label{sect:AppB}

 The sheafification functor $\mathfrak{a}_X$ for presheaves on a diffeological space is induced in Definition \ref{defn:top.sheafification}. We give the proof of Proposition \ref{prop:sheafification}.

\begin{proof}[Proof of theorem \ref{prop:sheafification}]
Let $P$ be an object of $PSh_{\mathsf{Diff}}^{\mathcal{C}}\mathcal{D}_{X}$ and $Q$ an object of $Sh_{\mathsf{Diff}}^{\mathcal{C}}\mathcal{D}_{X}$. 
We define a functor $g_{XP}^D : P \to \mathfrak{a}_{X}(P)$. Moreover, it is proved that for a map $\kappa_{X}^{D}:P\to \mathfrak{i}_{X}(Q)$ of $PSh_{\mathsf{Diff}}^{\mathcal{C}}\mathcal{D}_{X}$, there exist a unique map $\widetilde{\kappa}_{X}^{D}$ which makes the diagram 
\[
\xymatrix@C=25pt@R=10pt{
P\ar[rr]^{\kappa_{X}^{D}}\ar[rd]_{g^D_{XP}}&&\mathfrak{i}_{X}(Q)\\
&\mathfrak{a}_XP\ar[ru]_{\widetilde{\kappa}_{X}^{D}}\ar@{}|{\circlearrowright}[u]&
}
\]
commutative and a natural bijection $\text{Hom}_{Sh}(\mathfrak{a}_{X}(P),Q)\cong \text{Hom}_{PSh}(P,\mathfrak{i}_{X}(Q))$.

 Let  $\mathfrak{i}_{X}:Sh_{\mathsf{Diff}}^{\mathcal{C}}\mathcal{D}_{X}\to PSh_{\mathsf{Diff}}^{\mathcal{C}}\mathcal{D}_{X}$ be the inclusion functor defined by 
$\mathfrak{i}_{X}(Q)(p)=i_{U_{p}}(Q|_{p})(U_{p})$ for each plot $p\in \mathcal{D}_{X}$ and $Q\in Sh_{\mathsf{Diff}}^{\mathcal{C}}(X,\mathcal{D}_{X})$, where $i_{U_{p}}$ denotes the inclusion functor  $Sh_{\mathsf{Top}}^{\mathcal{C}}Open(U_p)\hookrightarrow PSh_{\mathsf{Top}}^{\mathcal{C}}Open(U_p)$. For a presheaf $P$,  
we define $g_{XP}^D : P \to \mathfrak{a}_{X}(P)$ by 
\[(g_{XP}^D)_p(s_p):=(g_{U_{p}P|_p}^T)_{U_p}(s_p)=([s_p]_u)_{u\in U_p}
\]
for each plot $p$ and $s_p\in P(p)(=P|_p(U_p))$, where $g_{U_p P|_p}^T$ is same with $g_{XP}$ in Theorem \ref{thm:univ.sheafification} but that the subscripts of the underlying set and the functor  are $U_p$ and $P|_p$ instead of $X$ and $P$, respectively. Then, for $f:p\to q$ in $\mathcal{D}_X$, we have commutative diagram  
$$
\xymatrix@R=20pt@C=50pt{
P(q)\ar[r]^(.4){(g_{XP}^D)_q}\ar[d]_{P(f)}&\mathfrak{a}_{X}P(q)\ar[d]^{\mathfrak{a}_XP(f)}\\
P(p)\ar[r]_(.4){(g_{XP}^D)_p}&\mathfrak{a}_{X}P(p)
}
$$
 which shows the map $g_{XP}^D$ is a natural transformation. The commutativity follows from the  construction of $g_{XP}^D$ and $\mathfrak{a}_XP(f)$; see Proposition  \ref{prop:well-def.sheafification}. Thus, $g_{XP}^{D}:P\to\mathfrak{a}_XP$ is a morphism of $PSh_{\mathsf{Diff}}^{\mathcal{C}}\mathcal{D}_X$. 
For a given morphism $\kappa_X^D$ on $PSh_{\mathsf{Diff}}^{\mathcal{C}}\mathcal{D}_X$ and each plot $p$, the map $(\kappa_{X}^{D})_{p}:P(p)\to\mathfrak{i}_X(Q)(p)$ gives rise to the map $(\kappa_{U_p}^T)_{U_p}:P|_{p}(U_{p})\to i_{U_{p}}(Q|_{p})(U_{p})$ in $\mathsf{Top}$. Then, for each open subset $W\subset U_{p}$, the same way as the assignment   $\kappa\mapsto\widetilde{\kappa}$ in Theorem \ref{thm:univ.sheafification} enable us to define $(\widetilde{\kappa}_{X}^{D})_{p|_{W}}$ with $(\widetilde{\kappa}_{U_p}^T)_{W}$ on $\mathsf{Top}$.

We have to check that $\widetilde{\kappa}_X^D$ is well defined. It is straightforward to see the naturality of $\widetilde{\kappa}_{X}^{D}$. Note that the construction of $((\kappa_{U_p}^T)_W\mapsto (\widetilde{\kappa}_{U_p}^T)_W)=((\kappa_{X}^D)_{p|_W}\mapsto (\widetilde{\kappa}_{X}^D)_{p|_W})$ is same as that  of  $(\lambda_p\mapsto(\mathfrak{a}_X\lambda)_p)$ in Proposition \ref{prop:sheafification0}. By construction of the map  $(\widetilde{\kappa}_X^D)_{p|_W}$, we see that  the  diagram 
$$
\xymatrix@C=20pt@R=5pt{
P(p|_{W})\ar[rr]^(.5){(\kappa_{X}^{D})_{p|_{W}}}\ar[rd]_(.4){(g^D_{XP})_{p|_W}}\ar[dd]_{=}&&\mathfrak{i}_{X}(Q)(p|_{W})\ar[dd]^{=}\\
&\mathfrak{a}_XP(p|_{W})\ar[ru]_{(\widetilde{\kappa}_{X}^{D})_{p|_{W}}}\ar[dd]^(.6){=}&\\
P|_{p}(W)\ar[rr]^(.3){(\kappa_{U_{p}}^{T})_{W}}\ar[rd]_(.4){(g_{U_p P|_p}^T)_W}&&i_{U_p}(Q|_{p})(W)\\
&a_{U_p}P|_{p}(W)\ar[ru]_{(\widetilde{\kappa}_{U_{p}}^{T})_{W}}&
}
$$
is commutative. Additionally, since  $\widetilde{\kappa}_{U_p}^T$ is unique, it follows that the map $(\widetilde{\kappa}^D_X)_p$ is a unique map for each $(\kappa_X^D)_p$. If there exists another map $(\widetilde{\kappa}'^D_{X})_p$ satisfying $(\widetilde{\kappa}'^D_{X})_p\circ (g_{XP}^D)_p=(\widetilde{\kappa}_X^D)_p$, we have $(\widetilde{\kappa}'^T_{U_p})_p$  with $(\widetilde{\kappa}'^T_{U_p})_{U_p}\circ (g_{U_p P|_p}^T)_{U_p}=(\kappa_{U_p}^T)_{U_p}$ naturally. But it contradicts the universality of $a_{U_p}$.  Then, $\widetilde{\kappa}^D_X$ is a unique map on $PSh_{\mathsf{Diff}}^{\mathcal{C}}\mathcal{D}_X$ with $\widetilde{\kappa}_X^D\circ g_{XP}^D=\kappa_X^D$.

We show the existence of a natural bijection $\text{Hom}_{PSh}(P,\mathfrak{i}_{X}(Q))\cong\text{Hom}_{Sh}(\mathfrak{a}_{X}(P),Q)$. Let $\theta^{D}_{X}: \text{Hom}_{PSh}(P,\mathfrak{i}_{X}(Q))\to \text{Hom}_{Sh}(\mathfrak{a}_{X}(P),Q)$ be the map defined by $\kappa_{X}^{D}\mapsto\widetilde{\kappa}^{D}_{X}$. 
For each $p\in\mathcal{D}_{X}$, the sheafification functor $a_{U_{p}}$ of presheaves on $U_p$ gives rise to a bijection $\theta_{U_{p}}^{T}:\text{Hom}_{PSh}(P|_{p},i_{U_p}(Q|_{p}))\to \text{Hom}_{Sh}(a_{U_{p}}(P|_{p}),Q|_{p})$. Let $\{\eta_{U_{p}}^{T}\}_{p\in\mathcal{D}_{X}}$ denote the collection of inverses of $\{\theta_{U_{p}}^{T}\}_{p\in\mathcal{D}_{X}}$. Observe that $\eta_{U_{p}}^{T}$ is defined by the composite with $g^T_{U_p P|_p}$ and a map of $\text{Hom}_{Sh}(a_{U_{p}}(P|_{p}),Q|_{p})$. Then, the commutative diagram 
\begin{eqnarray}\label{diag:theta}
\xymatrix@C=25pt@R=20pt{
\text{Hom}_{Sh}(\mathfrak{a}_{X}(P),Q)\ar@{^{(}-_>}[d]_{(2)}&\text{Hom}_{PSh}(P,\mathfrak{i}_{X}(Q))\ar[l]_{\theta^{D}_{X}}\ar@{^{(}-_>}[d]^{(1)}\\
\coprod_{p\in\mathcal{D}_{X}} \text{Hom}_{Sh}(a_{U_{p}}(P|_{p}),Q|_{p})\ar@<.5ex>[r]^{\coprod_{p}\eta^{T}_{U_{p}}}&\coprod_{p\in\mathcal{D}_{X}}\text{Hom}_{PSh}(P|_{p},i_{U_p}(Q|_{p}))\ar@<.5ex>[l]^{\coprod_{p}\theta^{T}_{U_{p}}}
}
\end{eqnarray}
enables us to deduce that the map $\theta_{X}^{D}$ is bijective. In fact, if there exist $\kappa_X^D$ and $\kappa'^D_X$  in $\text{Hom}_{Sh}(P,\mathfrak{i}_X(Q))$ such that  $\theta_X^D(\kappa_X^D)=\theta_X^D(\kappa'^D_X)$, through the inclusions $(1)$ and $(2)$, we have the equality $\{\kappa_{U_p}^T\}_{p\in\mathcal{D}_X}=\{\kappa'^T_{U_p}\}_{p\in\mathcal{D}_X}$. This yields that $\kappa_X^D=\kappa'^D_X$. We show the surjectivity of $\theta_X^D$. 
For a given $\sigma\in \text{Hom}_{Sh}(\mathfrak{a}_{X}^{D}(P),Q)$, we obtain the collection of morphisms $\{\sigma_p\}_{p\in\mathcal{D}_X}$ and it gives a family of maps $\{\sigma_{U_{p}}\}_{p}\in\coprod_{p\in\mathcal{D}_{X}} \text{Hom}_{Sh}(a_{U_{p}}(P|_{p}),Q|_{p})$, because $\sigma_{p}$ is identified with the morphism $\sigma_{U_p}:a_{U_{p}}(P|_{p})(U_{p})\to Q|_{p}(U_{p})$. It is easily seen that each element $\sigma_{U_p}$ is compatible with other components. Moreover, we have $\coprod_{p}\eta^{T}_{U_{p}}(\{\sigma_{U_{p}}\}_{p})=\{s_{U_{p}}:=g_{U_p P|_{p}}^T(\sigma_{U_p})\}_{p}\in\coprod_{p\in\mathcal{D}_{X}} \text{Hom}_{PSh}(P|_{p},i_{U_p}(Q|_{p}))$. 
The family of maps $\{s_{U_{p}}\}_{p}$ gives rise to a natural transformation $s\in \text{Hom}_{PSh}(P,\mathfrak{i}_{X}^{D}(Q))$. We have $\theta^{D}_{X}(s)=\sigma$. Finally, the naturalities for $P\in PSh_{\mathsf{Diff}}^{\mathcal{C}}\mathcal{D}_X$ and $Q\in Sh_{\mathsf{Diff}}^{\mathcal{C}}\mathcal{D}_X$ follows from that in the bottom sequence and the injections $(1)$ and $(2)$.
\end{proof}

\end{document}